\documentclass{article}
\usepackage[top=4cm, bottom=4cm, left=2.5cm, right=2.5cm]{geometry}
\usepackage{graphics}
 \usepackage{graphicx}
 \usepackage{epsfig}
 \usepackage{amsmath}
\usepackage{amsfonts}
\usepackage{amssymb}
\usepackage{xcolor}
\usepackage{enumerate}

\newtheorem{theorem}{Theorem}[section]
\newtheorem{prop}{Proposition}[section]
\newtheorem{leme}{Lemma}[section]
\newtheorem{dfnt}{Definition}[section]
\newtheorem{remark}{Remark}

\newenvironment{proof}[1][Proof]{\textbf{#1.} }{\ \rule{0.5em}{0.5em}}

\def \dis {\displaystyle}

\def \R {\mathbb{R}}

\def \O {\mathcal{O}}
\def\dt{dx\, dt}
\def\dT{dx\, dt}

\def \hvarphi \widehat{\varphi}

\def\dq{dx\,dt}

\def \hvarphi \widehat{\varphi}

\begin{document}
  \title{Stackelberg-Nash null controllability for a non linear coupled degenerate parabolic equations}
 \author{{{Landry Djomegne} \thanks{{\it
 				University of Dschang, BP 67 Dschang, Cameroon, West region,
 				email~: {\sf landry.djomegne\char64yahoo.fr} }}}
 	{{ \quad Cyrille Kenne} \thanks{{\it Laboratoire LAMIA, Universit\'e des Antilles, Campus Fouillole, 97159 Pointe-{\`a}-Pitre Guadeloupe (FWI),
 				email~: {\sf :kenne853\char64gmail.com} }}}
 	{{ \quad Ren\'e Dorville} \thanks{{\it
 				Laboratoire L3MA, UFR STE et IUT, Universit\'e des Antilles, Schoelcher, Martinique,
 				email~: {\sf rene.dorville\char64orange.fr} }}}
 			{{ \quad Pascal Zongo} \thanks{{\it
 						Laboratoire L3MA, UFR STE et IUT, Universit\'e des Antilles, Schoelcher, Martinique,
 						email~: {\sf pascal.zongo\char64gmail.com} }}}
 }

  \date{\today}
  
  \maketitle	
  	

\begin{abstract}
The main purpose of this paper is to apply the notion of hierarchical control to a coupled degenerate non linear parabolic equations. We use the Stackelberg-Nash strategy with one leader and two followers. The followers solve a Nash equilibrium corresponding to a bi-objective optimal control problem and the leader a null controllability problem. 
Since the considered problem is non linear, the associated cost is non-convex. We first prove the existence, uniqueness and the characterization of the Nash quasi-equilibrium, which is a weak formulation of the Nash equilibrium because the cost associated to the non linear problem is non-convex. Next, we show that under suitable conditions, the Nash quasi-equilibrium is equivalent to the Nash equilibrium. Finally using some Carleman inequalities that we established, and the Kakutani's fixed point Theorem, we brough the states of our system to the rest at final time $T$.

\end{abstract}


\section{Introduction}\paragraph{}

Let $\Omega=(0,1)$ be an open and bounded domain of $\R$. Let $\omega$, $\omega_1$ and $\omega_2$ be three non empty open subsets of $\Omega$ such that $\omega_i\cap\omega=\emptyset$, for $i=1,2$.  We fix $T > 0$ and set $Q =(0, T )\times \Omega$,
$\omega_T =(0, T )\times\omega $, $\omega_{1,T} =(0, T )\times\omega_1 $ and $\omega_{2,T} =(0, T )\times\omega_2 $. Then, we consider the following non linear coupled degenerate system

  \begin{equation}\label{eq}
  \left\{
  \begin{array}{rllll}
  \dis y_{1,t}-\left(a(x)y_{1,x}\right)_{x}+F_1(y_1)  &=&h\chi_{\omega}+v^1\chi_{\omega_1}+v^2\chi_{\omega_2}& \mbox{in}& Q,\\
  \dis y_{2,t}-\left(a(x)y_{2,x}\right)_{x}+F_2(y_2)+d\,y_1  &=&0& \mbox{in}& Q,\\
  \dis y_1(t,0)=y_1(t,1)=y_2(t,0)=y_2(t,1)&=&0& \mbox{on}& (0,T), \\
  \dis y_1(0,\cdot)=y_1^0,\ \ y_2(0,\cdot)=y_2^0&& &\mbox{in}&\Omega.
  \end{array}
  \right.
  \end{equation}
   
  In the system \eqref{eq}, $y=y(t,x)=(y_1,y_2)^{t}$ is the state, $v^i=v^i(t,x),\ i=1,2$ and $h=h(t,x)$ are different control functions  whose act on the system through the subsets $\omega_{i}$ and $\omega$ respectively. These functions $v^i$ and $h$ are the followers and leader controls respectively. Here $\chi_{\omega}$ and $\chi_{\omega_i}$ are respectively the characteristic function of the control set $\omega$ and $\omega_i$, $y^0=(y_1^0, y_2^0)^{t}\in [L^2(\Omega)]^2$ is the initial data and the function $d\in L^\infty(Q)$. 
  
  We assume that the real functions $a:=a(\cdot)$ and $F_i:\mathbb{R}\rightarrow \mathbb{R},\ i=1,2$ satisfy the following assumptions: 
  \begin{equation}\label{k}
  	\left\{\begin{array}{llll}
  		\dis a\in \mathcal{C}([0,1])\cap\mathcal{C}^1((0,1]),\ \ a>0\ \mbox{in}\ (0,1] \ \mbox{and}\ a(0)=0,\\
  		\dis \exists \tau\in [0,1)\ :\ xa^\prime(x)\leq \tau a(x),\ x\in [0,1]
  	\end{array}
  	\right.
  \end{equation}
  and
  \begin{equation}\label{lip}
  	\left\{
  	\begin{array}{rllll}
  		&&	\dis F_i(0)=0,\\
  		&&	\dis F_i\in \mathcal{C}^2(\mathbb{R}),\\
  		&& \dis \exists M>0:\ |F_i^\prime(r)|+|F_i^{\prime\prime}(r)|\leq M,\ \forall r\in \R,\ i=1,2.
  	\end{array}
  	\right.
  \end{equation}

Note that the above hypothesis on $a(\cdot)$ are true in the case where $a(x)=x^\alpha$ with $0\leq\alpha<1$. Then, in this case, the system \eqref{eq} will be called a weakly coupled degenerate system. We can also obtain the same results of this work in the case where $1\leq\alpha<2$ and this time we will rather take the Neumann condition $\left(a(x)y_x\right)(0)=0$ and the system \eqref{eq} will be called a strongly coupled degenerate system (\textit{cf}. \cite{alabau2006}).  We denote by $y_{1,t}$ and $y_{1,x}$ the partial derivative of $y_1$ with respect $t$ and $x$ respectively.\par
   
In the context of population dynamics, the system \eqref{eq} can models the dispersion of a gene in two given populations (cancer cells and healthy cells for instance) which are in interaction. In this case, $x$ represents the gene type, $y_1(t,x)$ and $y_2(t,x)$ denote the distributions of individuals at time $t$ and of gene type $x$ of both populations. In this paper, the function $a(x)$ is the diffusion coefficient which depends on the gene type and degenerate at the left hand side of its domain, i.e. $a(0)=0$, (e.g $a(x)=x^\alpha,\ \alpha>0$). In this case, we say that the system \eqref{eq} is a coupled degenerate parabolic equation. Genetically speaking, such a property of degeneracy is natural since it means that if each population is not of gene type, it cannot be transmitted to its offspring.\par

 In this paper we are interested in the hierarchic Stackelberg-Nash strategy for system \eqref{eq}. More precisely, for $i=1,2$, we introduce the non-empty open sets $\omega_{i,d}\subset \Omega$, representing the observation domains of the followers, and the fixed target functions $y_d^i=(y_{1,d}^i, y_{2,d}^i)^{t}\in L^2((0,T); \omega_{1,d})\times L^2((0,T); \omega_{2,d})$. Let us define 
  the following cost functional
  \begin{equation}\label{all16}
  J_i(h;v^1,v^2)=\frac{\alpha_i}{2}\int_0^T\int_{\omega_{i,d}}\left(|y_1-y_{1,d}^i|^2+|y_2-y_{2,d}^i|^2\right)\ dxdt+\frac{\mu_i}{2}\int_0^T\int_{\omega_i}\rho_{*}^2|v^i|^2\ dxdt,
  \end{equation}
  where  $\alpha_i$ and $\mu_i$ are two positive constants and $\rho_{*}=\rho_{*}(t)\in C^\infty([0,T])$ is a suitable positive weight function blowing up at $t=0$ and $t=T$.
  
  \begin{remark}\label{}$ $
  	
  	The weight function $\rho_{*}(t)$ defined in \eqref{all16} will help us to establish a suitable observability Carleman inequality in the Section \ref{Carleman}.
  \end{remark}

   We want to choose the controls $v^i$ and $h$ in order to achieve two different objectives:
   \begin{itemize}
   	\item The main goal is to choose $h$ such that the following null controllability objective holds:
   	\begin{equation}\label{obj1}
   	y_1(T,\cdot;h;v^1,v^2)=y_2(T,\cdot;h;v^1,v^2)=0\ \mbox{in}\ \Omega.
   	\end{equation}
   	\item The second goal is the following: given the functions $y_d^i$ and $h$, we want to choose the control $v^i$  minimizing $J_i$ given by \eqref{all16}. This means that, throughout the interval $(0,T)$,  the control $v^i$ will be chosen such that:
   	\begin{equation}\label{obj2}
   	\begin{array}{rll}
   	&&\mbox{the solution}\ y(t,x;h;v^1,v^2)\ \mbox{of}\ \eqref{eq}\ \mbox{remains "not too far" from a desired target}\  y_d^i(t,x) \\ 
   	&&\mbox{in the observability domain}\  \omega_{i,d},\ i=1,2.
   	\end{array}
   	\end{equation}
   \end{itemize} 
Our goal is to prove that, for any initial data $y^0\in [L^2(\Omega)]^2$, there exist a control $h\in L^2(\omega_T)$ (called leader) and an associated Nash equilibrium $(\hat{v}^1,\hat{v}^2)^{t}=(\hat{v}^1(h),\hat{v}^2(h))^{t}\in \mathcal{H}=L^2((0,T);L^2(\omega_1))\times L^2((0,T);L^2(\omega_2))$ (called followers) such that the associated state $y$ of system \eqref{eq} satisfies \eqref{obj1}. To do this, we shall follow the Stackelberg-Nash strategy which is described as follows:

\begin{enumerate}
\item For each choice of the leader $h$, we look for a Nash equilibrium pair for the costs $J_i,\ i=1,2$ given by \eqref{all16}. That is, find the controls $(\hat{v}^1,\hat{v}^2)^{t}=(\hat{v}^1(h),\hat{v}^2(h))^{t}\in \mathcal{H}$ satisfying 
 \begin{equation}\label{Nash}
	\left\{
	\begin{array}{rllll}
	\dis J_1(h;\hat{v}^1,\hat{v}^2)\leq J_1(h;v^1,\hat{v}^2),\ \ \ \forall v^1\in L^2((0,T);L^2(\omega_1)),\\
		\dis J_2(h;\hat{v}^1,\hat{v}^2)\leq J_2(h;\hat{v}^1,v^2),\ \ \ \forall v^2\in L^2((0,T);L^2(\omega_2)),
	\end{array}
	\right.
\end{equation}
or equivalently
 \begin{equation}\label{Nash1}
	\left\{
	\begin{array}{rllll}
		\dis	J_1(h;\hat{v}^1,\hat{v}^2)=\min_{v^1\in L^2((0,T);L^2(\omega_1))}J_1(h;v^1,\hat{v}^2),\\
		\dis	J_2(h;\hat{v}^1,\hat{v}^2)=\min_{v^2\in L^2((0,T);L^2(\omega_2))}J_2(h;\hat{v}^1,v^2).
	\end{array}
	\right.
\end{equation}

\item Once the Nash equilibrium has been identified and fixed for each $h$, we look for a control $\bar{h}$ such that
\begin{equation}\label{mainobj}
	y_1(T,\cdot;\bar{h};\hat{v}^1(\bar{h}),\hat{v}^2(\bar{h}))=y_2(T,\cdot;\bar{h};\hat{v}^1(\bar{h}),\hat{v}^2(\bar{h}))=0\ \mbox{in}\ \Omega.
\end{equation}
		
\end{enumerate}

\begin{remark}
	$ $
	\begin{enumerate}[(a)]
	\item In the linear case, the functionals $J_i,\ i=1,2$ are differentiable and convex and in this case,  the pair $(\hat{v}^1,\hat{v}^2)$ is a Nash equilibrium for $(J_1,J_2)$ if and only if
	\begin{equation}\label{Nashprime}
	\left\{
	\begin{array}{rllll}
	\dis	\frac{\partial J_1}{\partial v^1}(h;\hat{v}^1,\hat{v}^2)(v^1,0)=0,\ \ \forall v^1\in L^2((0,T);L^2(\omega_1)),\  \ \hat{v}^i\in L^2((0,T);L^2(\omega_i)),\\
	\dis	\frac{\partial J_2}{\partial v^2}(h;\hat{v}^1,\hat{v}^2)(0,v^2)=0,\ \ \forall v^2\in L^2((0,T);L^2(\omega_2)),\  \ \hat{v}^i\in L^2((0,T);L^2(\omega_i)).
	\end{array}
	\right.
	\end{equation}
	\item 	In the semi-linear framework, the corresponding functionals $J_1$ and $J_2$ are not convex in general. For this reason, we must consider the weaker definition of Nash equilibrium given below.	
	
	\end{enumerate}
	
\end{remark}

 \begin{dfnt}\label{dfntquasi}$ $
 	
 	Let the leader control $h$ be given. The pair ($\hat{v}^1,\hat{v}^2)$ is called a Nash quasi-equilibrium of functionals $(J_1,J_2)$ if the condition \eqref{Nashprime} is satisfied.		
 \end{dfnt}

In this paper, we are interested by the concept of Stackelberg competition introduced by \cite{Von1934Stackelberg}. It is a strategy  game between several firms in which one of the firms (called the leader) moves first and the others firms (named followers) moves according to the leader's strategy. In case of many followers with each corresponding to a specific optimality objective, the Nash equilibrium is the most suitable.

In the framework of partial differential equations (PDEs), the hierarchic control was introduced by J-L. Lions in \cite{Lions1994Stackelberg1, Lions1994Stackelberg2} to study a bi-objective control problem for the wave and heat equations respectively. In the last years, other authors have used hierarchical control in the sense of Lions, see for instance \cite{mercan1, mercan2, djomegnebackward, teresa2018, liliana2020, romario2018, djomegne2018, djomegnelinear}. There are in the literature some important results about
Stackelberg-Nash strategy for PDEs. In  \cite{araruna2015anash}, F. D Araruna et {\it al.} developed the first hierarchical results within the exact controllability framework for class of parabolic equations (linear and semi-linear), with pointwise constraints on the followers. In \cite{santos2019}, N. Carre\~{n}o and M. C. Santos applied the Stackelberg-Nash strategy to the Kuramoto-Sivashinsky equation with a distributed leader, and two followers. Their results were achieved by proving a partial null
controllability result for a system of non linear fourth-order equations with boundary coupling terms. In \cite{djomegne2021}, the author studied the Stackelberg-Nash strategy for a non linear parabolic equation in an unbounded domain. His results were achieved using a Schauder's fixed point Theorem under the assumption that the uncontrolled domain is bounded. In \cite{dany2021}, Dany Nina Huaman applied Stackelberg-Nash strategy  to control a quasi-linear parabolic equations in dimensions $1D$, $2D$ or $3D$. In \cite{omar2021}, the authors applied hierarchical control to the anisotropic heat equation with dynamic boundary conditions and drift terms. Even though the hierarchical control of several types of problems has been intensively considered by researchers (see e.g \cite{araruna2019boundary, ararunawave2018} ) there only a few addressing the important case of coupled systems. 

In the context of hierarchical strategy for coupled systems, K\'er\'e et {\it al.} \cite{kere2017coupled}, considered a bi-objective control strategy for a coupled parabolic equations with a finite constraints  on one of the states. Their results were
achieved by means of an observability inequality of Carleman adapted to
the constraints. In \cite{teresacoupled} and \cite{teresa}, V. Hern\'{a}ndez-Santamar\'ia et {\it al.} studied a Stackelberg-Nash strategy for a cascade system of parabolic equations and for a cascade system of parabolic equations with the leader as a vector function acting in the two equations. More recently in \cite{danynina2021}, J. Limaco et {\it al.} applied the Stackelberg-Nash strategy to a coupled  quasi-linear parabolic system with controls acting in the interior on the domain.      

In all the above cited works, the hierarchic strategy were applied to non degenerate systems. To the best of our knowledge, there do not exists any work in the literature addressing the Stackelberg-Nash strategy for a coupled degenerate parabolic system. As far as we know, the only work dealing with hierarchical strategy applied to a single degenerate equation is the one of F. D. Araruna et {\it al.} \cite{araruna2018stackelberg}, where the authors studied the Stackelberg-Nash strategy for a semi-linear degenerate parabolic equations. In the present paper, the main novelty is that, we extend the results concerning the Stackelberg-Nash control to the coupled degenerate parabolic equation \eqref{eq}. 

\subsection{Main results}

We will prove that if $\mu_i,\ i=1,2$ are sufficiently large, then the functionals $J_i,\ i=1,2$ given by \eqref{all16} are indeed convex. More precisely, we have the following results.

 \begin{theorem}\label{nash}$ $
 	
	Assume that \eqref{k}, \eqref{lip}, $y^0\in [L^2(\Omega)]^2$ and $y_d^i\in L^\infty((0,T); \omega_{1,d})\times L^\infty((0,T); \omega_{2,d})$ are satisfied. Assume that $h\in L^2(\omega_T)$ and $\mu_i,\ i=1,2$ are sufficiently large. 
	Then, if $(\hat{v}^1,\hat{v}^2)$ is a Nash quasi-equilibrium for $J_i,\ i=1,2$, there exists a constant $C>0$ independent of $\mu_i,\ i=1,2$ such 
	\begin{equation}
		\dis D_i^2J_i(h;\hat{v}^1,\hat{v}^2)\cdot(w^i,w^i)\geq C \|w^i\|^2_{L^2((0,T);L^2(\omega_i))},\ \forall w^i\in L^2((0,T);L^2(\omega_i)),\ i=1,2.
	\end{equation}
	In particular, the functional $(J_1,J_2)$ are convex in $(\hat{v}^1, \hat{v}^2)$ and therefore the pair $(\hat{v}^1, \hat{v}^2)$ is a Nash equilibrium for $J_i,\ i=1,2$ of the system \eqref{eq}.
	
\end{theorem}

 To state the main contribution of this paper, we assume that the control regions satisfy the following assumptions:
\begin{equation}\label{od}
	\left\{
	\begin{array}{rllll}
		&&\dis \omega_{1,d}=\omega_{2,d}: \mbox{the common observability set will be denoted by}\ \omega_d,\\
		\\
		&&\dis \omega_d\cap\omega\neq\emptyset.
	\end{array}
	\right.
\end{equation}

Our main result of this paper is the following:
\\
\begin{theorem}\label{theolinear}$ $
	
	 Suppose that \eqref{od} holds, the $\mu_i,\ i=1,2$ are large enough, $a(\cdot)$ satisfies \eqref{k} and the function $F_i,\ i=1,2$ satisfy \eqref{lip}. Let $\O_2$ be a non-empty subset of $\omega$ such that the following inequality holds:
	\begin{equation}\label{d}
		d\geq d_0>0\ \ \ \mbox{in}\ \ \ \ (0,T)\times \O_2.
	\end{equation}
Then,
	there exists a positive real weight function $\kappa=\kappa(t)$ (the definition of $\kappa$ will be given later) such that for any $y_d^i=(y_{1,d}^i, y_{2,d}^i)^{t}\in L^2((0,T); \omega_{1,d})\times L^2((0,T); \omega_{2,d})$ satisfying
	\begin{equation}\label{nou}
		\int_{0}^{T}\int_{\omega_d}\kappa^{-2}|y^i_{j,d}|^2\ dxdt<+\infty,\ \ i,j=1,2,
	\end{equation}
	and for any $y^0\in [L^2(\Omega)]^2$, there exist a control $\bar{h}\in L^2(\omega_T)$ and an associated Nash equilibrium $(\hat{v}^1, \hat{v}^2)\in \mathcal{H}=L^2((0,T);L^2(\omega_1))\times L^2((0,T);L^2(\omega_2))$ such that the corresponding solution to \eqref{eq} satisfies \eqref{mainobj}. 
\end{theorem}

 \begin{remark}
 		$ $
 		
  In this work, we assume that $\omega_i\cap\omega=\emptyset$. This means that the followers control cannot act on the leader's domain.

 \end{remark}

 The rest of this paper is organized as follows. In Section \ref{preliminary}, we give the proof of existence, uniqueness and characterization of Nash equilibrium. Section \ref{Carleman} deals with the proof of some suitable Carleman estimates. In Section \ref{null}, we deduce the null controllability result of system \eqref{eq}. Concluding remarks is made in Section \ref{conclusion}.

\section{Preliminary results}\label{preliminary}

 \subsection{Well-posedness result of system \eqref{eq}}\label{well}
 
 In the sequel, the usual norm in $L^\infty(Q)$ will be denoted by $\|\cdot\|_{\infty}$.
 In order to study the well-posedness of system \eqref{eq}, we introduce as in \cite{cannarsa2005, cannarsa2008, cannarsa2016} the weight spaces $H^1_{a}(\Omega)$ and $H^2_{a}(\Omega)$ as follows (in the sequel, abs. cont. means absolutely continuous):
  
 \begin{equation}\label{}
 \left\{\begin{array}{llll}
 \dis H^1_{a}(\Omega)=\{u\in L^2(\Omega): u\ \mbox{is abs. cont. in}\ [0,1]: \sqrt{a}u_x\in L^2(\Omega),\ u(0)=u(1)=0\},\\
 \dis H^2_{a}(\Omega)=\{u\in H^1_{a}(\Omega): a(x)u_x\in H^1(\Omega)\},
 \end{array}
 \right.
 \end{equation}
 endowed respectively with the norms
 \begin{equation}\label{}
 \left\{\begin{array}{llll}
 \dis \|u\|^2_{H^1_{a}(\Omega)}=\|u\|^2_{L^2(\Omega)}+\|\sqrt{a}u_x\|^2_{L^2(\Omega)},\ \ \ u\in H^1_{a}(\Omega),\\
 \dis \|u\|^2_{H^2_{a}(\Omega)}=\|u\|^2_{H^1_{a}(\Omega)}+\|(a(x)u_x)_x\|^2_{L^2(\Omega)},\ \ \ u\in H^2_{a}(\Omega).
 \end{array}
 \right.
 \end{equation}
	Under the assumptions \eqref{k}, we have the following embedding (see \cite[Lemma 2.2]{floridia2015})
\begin{equation}\label{embedd}
	H^1_a(\Omega)\hookrightarrow L^\infty(\Omega).
	\end{equation}

 Proceeding as in \cite[Proposition 2.3]{maniar2011} by using the  semi-group  theory or using the variational approach \cite[Theorem 1.1, page 37]{Lions1961}, we can prove the following result.  
 
\begin{prop}\label{exis}$ $
	
	Assume that \eqref{k} and \eqref{lip} are valid. Let $y^0\in [L^2(\Omega)]^2$, $h\in L^2(\omega_T)$ and   $(v^1, v^2)^{t}\in \mathcal{H}=L^2((0,T);L^2(\omega_1))\times L^2((0,T);L^2(\omega_2))$. Then, the system \eqref{eq} admits a unique weak solution 
	\begin{equation}
		y\in\mathbb{H}=L^2((0,T);[H^1_{a}(\Omega)]^2)\cap[\mathcal{C}([0,T];L^2(\Omega))]^2.
	\end{equation}
	Moreover, there exists a constant $C=C(T,K)>0$ such that the following estimation holds:
	\begin{equation}\label{esty1y2}
		\begin{array}{llllll}
			\dis \sup_{\tau\in [0,T]}\|y_1(\tau,\cdot)\|^2_{L^2(\Omega)}+\sup_{\tau\in [0,T]}\|y_2(\tau,\cdot)\|^2_{L^2(\Omega)}+\|y_1\|^2_{L^2((0,T); H^1_a(\Omega))}+\|y_2\|^2_{L^2((0,T); H^1_a(\Omega))}\\\dis 
			\\
			\leq 
			C\left(\|v^1\|^2_{L^2(\omega_{1,T})}+\|v^2\|^2_{L^2(\omega_{2,T})}+\|h\|^2_{L^2(\omega_T)}+\|y_1^0\|^2_{L^2(\Omega)}+\|y_2^0\|^2_{L^2(\Omega)}\right),
		\end{array}
	\end{equation} 	
where $K$ is the Lipschitz constant.
\end{prop}
 
We state the Hardy-Poincar\'e inequality that will be useful for the rest of the paper. This inequality is similar to the one stated on \cite[Proposition 2.1]{alabau2006} as well as its proof.

\begin{prop}(Hardy-Poincar\'e inequality) $ $
	
	Assume that $a:[0;1]\longrightarrow\R_+$ is in $\mathcal{C}([0;1])$, $a(0)=0$ and $a>0$ on $(0;1]$. Furthermore, assume that $a$ is such that there exists $\theta\in (0;1)$ such that the function $\dis x\longmapsto\frac{a(x)}{x^\theta}$ is non-increasing in a neighbourhood of zero. Then, there is a constant $\overline{C}>0$ such that for any $z$, locally absolutely continuous on $(0;1]$, continuous at $0$ and satisfying $z(0)=0$ and $\dis \int_0^1 a(x)|z^\prime(x)|^2\ dx<+\infty$, the following inequality holds
	
	\begin{equation}\label{hardy}
	\int_0^1 \frac{a(x)}{x^2}|z(x)|^2\ dx<\overline{C}\int_0^1 a(x)|z^\prime(x)|^2\ dx.	
	\end{equation}
	Moreover, under the same hypothesis on $z$ and the fact that the function $\dis x\longmapsto\frac{a(x)}{x^\theta}$ is non-increasing on $(0;1]$, then the inequality \eqref{hardy} holds with $\dis \overline{C}=\frac{4}{(1-\theta)^2}.$	
\end{prop}

\subsection{Characterization of Nash equilibrium}\label{low}

First, we give the following characterization of Nash quasi-equilibrium pair (recall Definition \ref{dfntquasi}) $(\hat{v}^1,\ \hat{v}^2)$ of \eqref{eq} for $J_i,\ i=1,2$ given by \eqref{all16}.

\begin{prop}\label{quasi}$ $
	
Let $h\in L^2(\omega_T)$ and assume that $\mu_i,\ i=1,2$ are sufficiently large. Let also $(\hat{v}^1,\ \hat{v}^2)\in \mathcal{H}=L^2((0,T);L^2(\omega_1))\times L^2((0,T);L^2(\omega_2))$ be the Nash quasi-equilibrium pair for $(J_1,J_2)$. Then, there exists $p^i=(p_1^i, p_2^i)\in \mathbb{H}$ such that Nash quasi-equilibrium pair $(\hat{v}^1, \hat{v}^2)\in \mathcal{H}$ is characterized by
\begin{equation}\label{vop}
	\hat{v}^i=-\frac{1}{\mu_i}\rho_{*}^{-2}p_1^i\ \ \mbox{in}\ \ \ (0,T)\times \omega_i,	
\end{equation}
where $y=(y_1,y_2)$ and $p^i=(p_1^i,p_2^i)$ are solutions of the following optimality systems
\begin{equation}\label{yop}
	\left\{
	\begin{array}{rllll}
		\dis y_{1,t}-\left(a(x)y_{1,x}\right)_{x}+F_1(y_1)  &=&\dis h\chi_{\omega}-\frac{1}{\mu_1}\rho_{*}^{-2}p_1^1\chi_{\omega_1}-\frac{1}{\mu_2}\rho_{*}^{-2}p_1^2\chi_{\omega_2}& \mbox{in}& Q,\\
		\dis y_{2,t}-\left(a(x)y_{2,x}\right)_{x}+F_2(y_2)+dy_1  &=&0& \mbox{in}& Q,\\
		\dis y_1(t,0)=y_1(t,1)=y_2(t,0)=y_2(t,1)&=&0& \mbox{on}& (0,T), \\
		\dis y_1(0,\cdot)=y_1^0,\ \ y_2(0,\cdot)=y_2^0&& &\mbox{in}&\Omega
	\end{array}
	\right.
\end{equation}
and 
\begin{equation}\label{pop}
	\left\{
	\begin{array}{rllll}
		\dis -p_{1,t}^i-\left(a(x)p^i_{1,x}\right)_{x}+F^\prime_1(y_1)p_1^i+dp_2^i   &=&\alpha_i\left(y_1-y_{1,d}^i\right)\chi_{\omega_{i,d}}& \mbox{in}& Q,\\
		\dis -p_{2,t}^i-\left(a(x)p^i_{2,x}\right)_{x}+F^\prime_2(y_2)p_2^i  &=&\alpha_i\left(y_2-y_{2,d}^i\right)\chi_{\omega_{i,d}}& \mbox{in}& Q,\\
		\dis p_1^i(t,0)=p_1^i(t,1)=p_2^i(t,0)=p_2^i(t,1)&=&0& \mbox{on}& (0,T), \\
		\dis p_1^i(T,\cdot)= p_2^i(T,\cdot)&=&0 &\mbox{in}&\Omega.
	\end{array}
	\right.
\end{equation}

\end{prop}

\begin{proof}
	
If $(\hat{v}^1,\ \hat{v}^2)$ is a Nash quasi-equilibrium in the sense of the Definition \ref{dfntquasi}, then we have
\begin{equation}\label{EL}
	\begin{array}{rll}
		&&\dis 	\alpha_i\int_0^T\int_{\omega_{i,d}}\left[(y_1-y_{1,d}^i)z_1^i+(y_2-y_{2,d}^i)z_2^i\right]\ \dT\\
		&&\dis +\mu_i\int_0^T\int_{\omega_i}\rho_{*}^{2}\hat{v}^iv^i\ \dT=0,\ \mbox{for all}\  v^i\in L^2((0,T);L^2(\omega_i)),
	\end{array}
\end{equation}
where $z^i=(z_1^i, z_2^i)$ is the solution of the following system
\begin{equation}\label{zop}
	\left\{
	\begin{array}{rllll}
		\dis z_{1,t}^i-\left(a(x)z^i_{1,x}\right)_{x}+F^\prime_1(y_1)z_1^i   &=&v^i\chi_{\omega_{i}}& \mbox{in}& Q,\\
		\dis z_{2,t}^i-\left(a(x)z^i_{2,x}\right)_{x}+F^\prime_2(y_2)z_2^i+dz_1^i  &=&0& \mbox{in}& Q,\\
		\dis z_1^i(t,0)=z_1^i(t,1)=z_2^i(t,0)=z_2^i(t,1)&=&0& \mbox{on}& (0,T), \\
		\dis z_1^i(0,\cdot)= z_2^i(0,\cdot)&=&0 &\mbox{in}&\Omega.
	\end{array}
	\right.
\end{equation}
If we multiply \eqref{pop} by $z^i$ solution of \eqref{zop} and integrate by parts on $Q$, we obtain
\begin{equation*}
	\begin{array}{rll}
		\dis 	\alpha_i\int_0^T\int_{\omega_{i,d}}\left[(y_1-y_{1,d}^i)z_1^i+(y_2-y_{2,d}^i)z_2^i\right]\ \dT =\int_0^T\int_{\omega_i}v^ip_1^i\ \dT=0.
	\end{array}
\end{equation*}
Combining this latter equality with \eqref{EL}, we obtain
\begin{equation*}
	\begin{array}{rll}
		\dis \int_0^T\int_{\omega_i}v^i\left(p_1^i+\mu_i\rho_{*}^{2}\hat{v}^i\right)\ \dT=0,\ \mbox{for all}\  v^i\in L^2((0,T);L^2(\omega_i)),
	\end{array}
\end{equation*}
from where \eqref{vop}-\eqref{pop} follows.
\end{proof}

 \begin{remark}
 	$ $
 	\begin{enumerate}
 		\item Notice that the existence and uniqueness of a solution for \eqref{yop}-\eqref{pop} implies the existence and uniqueness of a Nash quasi-equilibrium in the sense of Definition \ref{dfntquasi}. Proposition \ref{exis} guarantees the existence and uniqueness of solution for system \eqref{yop}-\eqref{pop}.
 		
 		\item Using the idea of  \cite[Proposition $2.1$]{djomegne2021}, we can prove the existence of a constant $C>0$ such that
 		\begin{equation}\label{v}
 		\|(\hat{v}^1,\hat{v}^2)\|_{\mathcal{H}}\leq 
 		C\left(1+\|h\|_{L^2(\omega_T)}\right).
 		\end{equation} 	
 	\end{enumerate}	
 \end{remark}

Now, we introduce the following results useful for proving the Theorem \ref{nash}.
	We consider the system
\begin{equation}\label{modeli}
\left\{
\begin{array}{rllll}
\dis z_{t}-\left(a(x)z_{x}\right)_{x}+c(t,x)z &=&g& \mbox{in}& Q,\\
\dis z(\cdot,0)=z(\cdot,1)&=&0& \mbox{in}& (0,T), \\
\dis  z(0,\cdot)&=&z^0 &\mbox{in}&\Omega.
\end{array}
\right.
\end{equation}
We set \begin{equation}\label{defk}
\mathbb{K}:=H^1((0,T);L^2(\Omega))\cap L^2((0,T);H^2_a(\Omega))\cap \mathcal{C}([0,T];H^1_a(\Omega)).
\end{equation}
Then we have the following result proved in \cite[Theorem 2.1]{alabau2006}.
\begin{theorem}\label{theoint}
	Let $z^0\in H^1_a(\Omega)$, $c\in L^\infty(Q)$ and $g\in L^2(Q)$. Then, the system \eqref{modeli} admits a
	unique weak solution $z\in	\mathbb{K} $. Moreover there exists a positive constant $C=C(T)$ such that
	\begin{equation}\label{estimationint}
	\begin{array}{llll}
	\dis\sup_{t\in [0,T]}\|z(t)\|^2_{H^1_a(\Omega)}+\int_{0}^{T}\left(\|z_t\|^2_{L^2(\Omega)}+\|(a(x)z_x)_x\|^2_{L^2(\Omega)}\right)\, \dt
	\leq C(\|z^0\|^2_{H^1_a(\Omega)}+\|g\|^2_{L^2(Q)}).
	\end{array}
	\end{equation}	
\end{theorem}
\begin{leme}\label{sauveur}
	Let $z^0\in H^1_a(\Omega)$, $c\in L^\infty(Q)$ and $g\in L^2(Q)$. Then the
unique weak solution $z\in	\mathbb{K} $ of \eqref{modeli} belongs to $L^\infty(Q)$. Moreover there exists a positive constant $C=C(T)$ such that
\begin{equation}\label{estimation}
\begin{array}{llll}
\|z\|_{L^\infty(Q)}
\leq C(\|z^0\|_{H^1_a(\Omega)}+\|g\|_{L^2(Q)}).
\end{array}
\end{equation}	
\end{leme}
\begin{proof} Let $t\in [0,T]$, then by \eqref{estimationint}, we have 
	\begin{equation}\label{001}
	\sup_{t\in [0,T]}\|z(t)\|_{H^1_a(\Omega)}\leq C\left[\|y^0\|_{H^1_a(\Omega)}+\|g\|_{L^2(Q)}\right].
	\end{equation}
	Using the embedding \eqref{embedd}, we obtain for almost every $t\in [0,T]$,
	\begin{equation*}
	\|z(t)\|_{L^\infty(\Omega)}\leq C\|z(t)\|_{H^1_a(\Omega)}.
	\end{equation*}	
	Therefore
	\begin{equation}\label{002}
	\|z\|_{L^\infty(Q)}\leq C\sup_{t\in [0,T]}\|z(t)\|_{H^1_a(\Omega)}.
	\end{equation}	
	Combining \eqref{001} and \eqref{002}, we deduce \eqref{estimation}.
\end{proof}

Next, we prove Theorem \ref{nash} establishing the equivalence between Nash quasi-equilibrium and Nash equilibrium in the semi-linear case. The technique of the proof is inspired by \cite[Proposition 1.4]{araruna2015anash}. \paragraph{}

\textbf{Proof of Theorem \ref{nash}.} 

 Let $h\in L^2(\omega_T)$ be given and let $(\hat{v}_1, \hat{v}_2)$ be the associated Nash quasi-equilibria. For any $w^1,w^2\in L^2(\omega_{1,T})$ and $s\in \R$, let us de
 	note by $y^s=(y_1^s,y_2^s)$ the solution of the following system
 	\begin{equation}\label{ys}
 	\left\{
 	\begin{array}{rllll}
 	\dis y^s_{1,t}-\left(a(x)y^s_{1,x}\right)_{x}+F_1(y_1^s)  &=&h\chi_{\omega}+(v^1+sw^1)\chi_{\omega_1}+v^2\chi_{\omega_2}& \mbox{in}& Q,\\
 	\dis y^s_{2,t}-\left(a(x)y^s_{2,x}\right)_{x}+F_2(y^s_2)+dy^s_1  &=&0& \mbox{in}& Q,\\
 	\dis y^s_1(t,0)=y^s_1(t,1)=y^s_2(t,0)=y^s_2(t,1)&=&0& \mbox{on}& (0,T), \\
 	\dis y^s_1(0,\cdot)=y_1^0,\ \ y^s_2(0,\cdot)=y_2^0&& &\mbox{in}&\Omega
 	\end{array}
 	\right.
 	\end{equation}
 	and let us set $y_i:=y_i^s|_{s=0}$.
 	
 	Now, we have
 	\begin{equation}\label{dj}
 	\begin{array}{llll}
 	&&\dis D_1J_1(h;\hat{v}^1+sw^1,\hat{v}^2)\cdot w^2-	D_1J_1(h;\hat{v}^1,\hat{v}^2)\cdot w^2=s\mu_1\int_0^T\int_{\omega_{1}}\rho_{*}^{2}w^1w^2\ dxdt\\
 	&&\dis +\alpha_1\int_0^T\int_{\omega_{1,d}}\left[(y_1^s-y^1_{1,d})z_1^{1,s}+(y_2^s-y^1_{2,d})z_2^{1,s}\right]\ dxdt\\
 	&&\dis -\alpha_1\int_0^T\int_{\omega_{1,d}}\left[(y_1-y^1_{1,d})z_1^1+(y_2-y^1_{2,d})z_2^1\right]\ dxdt,
 	\end{array}
 	\end{equation}
 	where $z^{1,s}=(z_1^{1,s},z_2^{1,s})$ is the derivative of the state $y_i^s$ with respect to $\hat{v}^1+sw^1$ in the direction $w^2$, i.e. $z^{1,s}$ is the solution to
 	\begin{equation}\label{zs}
 	\left\{
 	\begin{array}{rllll}
 	\dis z_{1,t}^{1,s}-\left(a(x)z^{1,s}_{1,x}\right)_{x}+F^\prime_1(y_1^s)z_1^{1,s}   &=&w^2\chi_{\omega_{1}}& \mbox{in}& Q,\\
 	\dis z_{2,t}^{1,s}-\left(a(x)z^{1,s}_{2,x}\right)_{x}+F^\prime_2(y_2^s)z_2^{1,s}+dz_1^{1,s}  &=&0& \mbox{in}& Q,\\
 	\dis z_1^{1,s}(t,0)=z_1^{1,s}(t,1)=z_2^{1,s}(t,0)=z_2^{1,s}(t,1)&=&0& \mbox{on}& (0,T), \\
 	\dis z_1^{1,s}(0,\cdot)= z_2^{1,s}(0,\cdot)&=&0 &\mbox{in}&\Omega,
 	\end{array}
 	\right.
 	\end{equation}
 	and we have used the notation $z_i^1:=z_i^{1,s}|_{s=0},\ i=1,2$.
 	
 	Let us introduce the adjoint of \eqref{zs}
 	\begin{equation}\label{ps}
 	\left\{
 	\begin{array}{rllll}
 	\dis -p_{1,t}^{1,s}-\left(a(x)p^{1,s}_{1,x}\right)_{x}+F^\prime_1(y_1^s)p_1^{1,s}+dp_2^{1,s}   &=&\alpha_1\left(y_1^s-y_{1,d}^1\right)\chi_{\omega_{1,d}}& \mbox{in}& Q,\\
 	\dis -p_{2,t}^{1,s}-\left(a(x)p^{1,s}_{2,x}\right)_{x}+F^\prime_2(y_2^s)p_2^{1,s}  &=&\alpha_1\left(y_2^s-y_{2,d}^1\right)\chi_{\omega_{1,d}}& \mbox{in}& Q,\\
 	\dis p_1^{1,s}(t,0)=p_1^{1,s}(t,1)=p_2^{1,s}(t,0)=p_2^{1,s}(t,1)&=&0& \mbox{on}& (0,T), \\
 	\dis p_1^{1,s}(T,\cdot)= p_2^{1,s}(T,\cdot)&=&0 &\mbox{in}&\Omega.
 	\end{array}
 	\right.
 	\end{equation}
 	and let us use the notation $p_i^1:=p_i^{1,s}|_{s=0}$.\\
 	Multiplying the first  and the second equation of \eqref{zs} by $p_1^{1,s}$ and $p_2^{1,s}$, respectively and integrating by parts over $Q$, we obtain
 	\begin{equation}\label{zp}
 	\alpha_1\int_0^T\int_{\omega_{1,d}}\left[(y_1^s-y^1_{1,d})z_1^{1,s}+(y_2^s-y^1_{2,d})z_2^{1,s}\right]\ dxdt=\int_{Q}w^2p_1^{1,s}\chi_{\omega_1}\ dxdt.	
 	\end{equation}
 	From \eqref{dj} and \eqref{zp}, we have
 	\begin{equation}\label{dj1}
 	\begin{array}{llll}
 	\dis D_1J_1(h;\hat{v}^1+sw^1,\hat{v}^2)\cdot w^2-	D_1J_1(h;\hat{v}^1,\hat{v}^2)\cdot w^2&=&\dis s\mu_1\int_0^T\int_{\omega_{1}}\rho_{*}^{2}w^1w^2\ dxdt\\
 	\dis &+&\dis \int_0^T\int_{\omega_{1}}(p_1^{1,s}-p_1^1)w^2\ dxdt.
 	\end{array}
 	\end{equation}
 	Note that 
 	
 	\begin{equation*}
 	\left\{
 	\begin{array}{rllll}
 	\dis (y_1^{s}-y_1)_t-\left(a(x)(y_1^{s}-y_1)_x\right)_{x} +\left[F_1(y_1^s)-F_1(y_1)\right]=sw^1\chi_{\omega_1},\\
 	\\
 	\dis (y_2^{s}-y_2)_t-\left(a(x)(y_2^{s}-y_2)_x\right)_{x} +\left[F_2(y_2^s)-F_2(y_2)\right]+d(y_1^{s}-y_1)=0
 	\end{array}
 	\right.
 	\end{equation*}
 	and
 	\begin{equation*}
 		\left\{
 		\begin{array}{rllll}
 				\dis -(p_1^{1,s}-p_1^{1})_t-\left(a(x)(p_1^{1,s}-p_1^{1})_x\right)_{x} +\left[F_1^\prime(y_1^s)-F_1^\prime(y_1)\right]p_1^{1,s}+F_1^\prime(y_1)p_1^{1,s}-p_1^{1}+d(p_2^{1,s}-p_2^{1}),\\
 			\dis=\alpha_1(y_1^s-y_{1})\chi_{\omega_{1,d}},\\
 			\dis -(p_2^{1,s}-p_2^{1})_t-\left(a(x)(p_2^{1,s}-p_2^{1})_x\right)_{x} +\left[F_2^\prime(y_2^s)-F_2^\prime(y_2)\right]p_2^{1,s}+F_2^\prime(y_2)p_2^{1,s}-p_2^{1}
 			\dis=\alpha_1(y_2^s-y_{2})\chi_{\omega_{1,d}}.
 		\end{array}
 		\right.
 	\end{equation*}
 Consequently, using that $F\in \mathcal{C}^2(\R)$, we obtain that the following limits 
 	$$
 	\eta_i^{1}=\lim\limits_{s\to 0}\frac{1}{s}(p_i^{1,s}-p_i^{1})\ \ \mbox{and}\ \ \phi_i=\lim\limits_{s\to 0}\frac{1}{s}(y_i^s-y_i),\ \mbox{for}\ i=1,2
 	$$
 	exist and satisfy
 	\begin{equation}\label{eta}
 	\left\{
 	\begin{array}{rllll}
 	\dis -\eta^{1}_{1,t}-\left(a(x)\eta^{1}_{1,x}\right)_{x}+F_1^\prime(y_1)\eta_1^{1}+F^{\prime\prime}_1(y_1)\phi_1p_1^{1}+d\eta_2^{1}   &=&\alpha_1\phi_1\chi_{\omega_{1,d}}& \mbox{in}& Q,\\
 	\dis -\eta^{1}_{2,t}-\left(a(x)\eta^{1}_{2,x}\right)_{x}+F_2^\prime(y_2)\eta_2^{1}+F^{\prime\prime}_1(y_2)\phi_2p_2^{1}  &=&\alpha_1\phi_2\chi_{\omega_{1,d}}& \mbox{in}& Q,\\
 	\dis \eta_1^{1}(t,0)=\eta^{1}_1(t,1)=\eta^{1}_2(t,0)=\eta^{1}_2(t,1)&=&0& \mbox{on}& (0,T), \\
 	\dis \eta^{1}_1(T,\cdot)= \eta^{1}_2(T,\cdot)&=&0 &\mbox{in}&\Omega
 	\end{array}
 	\right.
 	\end{equation}
 	and
 	\begin{equation}\label{ha}
 	\left\{
 	\begin{array}{rllll}
 	\dis \phi_{1,t}-\left(a(x)\phi_{1,x}\right)_{x}+F_1^\prime(y_1)\phi_1 &=&w^1\chi_{\omega_1}& \mbox{in}& Q,\\
 	\dis \phi_{2,t}-\left(a(x)\phi_{2,x}\right)_{x}+F_2^\prime(y_2)\phi_2+d\phi_1  &=&0& \mbox{in}& Q,\\
 	\dis \phi_1(t,0)=\phi_1(t,1)=\phi_2(t,0)=\phi_2(t,1)&=&0& \mbox{on}& (0,T), \\
 	\dis \phi_1(0,\cdot)= \phi_2(0,\cdot)&=&0 &\mbox{in}&\Omega.
 	\end{array}
 	\right.
 	\end{equation}

 	Thus, from \eqref{dj1}-\eqref{ha} for $w^2=w^1$, we have 
 	\begin{equation}\label{dj3}
 	\dis D_1^2J_1(h;\hat{v}^1,\hat{v}^2)\cdot(w^1,w^1)= \int_0^T\int_{\omega_{1}}\rho_{*}^2\eta_1^{1} w^1\ dxdt+\mu
 	_1\int_0^T\int_{\omega_{1}}|w^1|^2\ dxdt.
 	\end{equation}
 	
 	Let us show that, for some constant $C>0$ independent of $h,\ \eta^1,\ \phi,\ w^1$, one has
 	\begin{equation}\label{a}
 	\left|\int_0^T\int_{\omega_{1}}\rho_{*}^2\eta_1^{1} w^1\ dxdt\right|\leq C\|w^1\|^2_{L^2(\omega_{1,T})}.
 	\end{equation}
 	The energy estimates associated to systems \eqref{ha} and \eqref{eta} are given respectively by
 	\begin{equation}\label{esta}
 	\| \phi_1\|^2_{L^2(Q)}+\| \phi_2\|^2_{L^2(Q)} \leq C\|w^1\|^2_{L^2(\omega_{1,T})}
 	\end{equation}
 and 
 \begin{equation}\label{etaest}
 	\| \eta_1^1\|^2_{L^2(Q)}+\| \eta_2^2\|^2_{L^2(Q)} \leq C\|w^1\|^2_{L^2(\omega_{1,T})}.
 \end{equation}
 	
 	As $(\hat{v}^1, \hat{v}^2)$ is a Nash quasi-equilibrium, $\hat{v}^i\in L^2((0,T);\omega_i)$ and $y_d^i\in L^\infty((0,T); \omega_{1,d})\times L^\infty((0,T); \omega_{2,d})$, then, $p^1=(p_1^1,p_2^1)$ solution of \eqref{ps} belongs to $[L^2((0,T);H^1_{a}(\Omega))]^2$ and using energy estimates, we obtain 
 	\begin{equation}\label{estb}
 	\| p_1^1\|^2_{L^2(Q)}+\| p_2^1\|^2_{L^2(Q)} \leq C,
 	\end{equation} 
 	where $C$ is a positive constant which is independent of $\mu_1$ and $\mu_2$.
 	In addition with the change of variables $t\mapsto T-t$ and applying Lemma \ref{sauveur} with $c=F'_1(y_1)\in L^\infty(Q)$ and $g=\alpha_1(y_1-y^1_{1,d})\chi_{\omega_{1,d}}-dp^1_2\in L^2(Q)$ and combining the result with \eqref{estb}, we deduce that
  \begin{equation}\label{inf}
  	\|p_1^1\|_{L^\infty(Q)}\leq C,
\end{equation} 
 	with $C$ a positive constant independent of $\mu_1$ and $\mu_2$.
 	Using systems \eqref{eta}-\eqref{ha} and the estimates \eqref{esta}-\eqref{inf}, we obtain 
 	\begin{equation}\label{fa1}
 	\begin{array}{llll}
 	\dis \int_{\omega_{1,T}}\rho_{*}^2\eta_1^1 w^1\ dxdt
 	&=&\dis \int_Q\rho_{*}^2\left(\alpha_1|\phi_1|^2\chi_{\omega_d}-F_1^{\prime\prime}(y_1)|\phi_1|^2 p_1^1-d\eta_2^1\phi_1\right)\ dxdt\\
 	&\leq&\dis \alpha_1\|\phi_1\|^2_{L^2(Q)}+ \|F_1^{\prime\prime}\|_{\infty}\|\phi_1\|^2_{L^{2}(Q)} \|p_1^1\|_{L^{\infty}(Q)}+ \|d\|_{\infty}\|\eta_2^1\|_{L^{2}(Q)} \|\phi_1\|_{L^{2}(Q)}\\
 	&\leq& C\|w^1\|^2_{L^2(\omega_{1,T})},
 	\end{array}
 	\end{equation}
 	where $C$ is independent of $\mu_1$ and $\mu_2$.
 	
 	We can use \eqref{a} in \eqref{dj3} and we obtain 
 	\begin{equation*}
 	\dis D_1^2J_1(h;\hat{v}^1,\hat{v}^2)\cdot(w^1,w^1)\geq\left(\mu_1-C\right) \int_{\omega_{1,T}}|w^1|^2\ dxdt,\ \forall w^1\in L^2(\omega_{1,T}).
 	\end{equation*}
 	In a similar way, we can prove that there exists a positive constant $C$ independent of $\mu_1$ and $\mu_2$ such that
 	\begin{equation*}
 	\dis D_2^2J_2(h;\hat{v}^1,\hat{v}^2)\cdot(w^2,w^2)\geq\left(\mu_2-C\right) \int_{\omega_{2,T}}|w^2|^2\ dxdt,\ \forall w^2\in L^2(\omega_{2,T}).
 	\end{equation*}
 	Taking $\mu_i$ sufficiently large, then the functional $J_i,\ i=1,2$ given by \eqref{all16} is convex and therefore the pair $(\hat{v}^1,\hat{v}^2)$ is a Nash equilibrium in the sense of \eqref{Nash1}. \hfill $\blacksquare$

\section{Carleman estimates}\label{Carleman}

In this section we establish an observability inequality that allows us to prove the null controllability of system \eqref{yop}-\eqref{pop}. We start by proving an observability inequality for the adjoint systems associated to the linearized version of \eqref{yop}-\eqref{pop} 

\begin{equation}\label{rho}
	\left\{
	\begin{array}{rllll}
		\dis -\rho_{1,t}-\left(a(x)\rho_{1,x}\right)_{x}+b_1\rho_1+d\rho_2  &=&\dis \alpha_1\psi_1^1\chi_{\omega_{1,d}}+\alpha_2\psi_1^2\chi_{\omega_{2,d}}& \mbox{in}& Q,\\
		\dis -\rho_{2,t}-\left(a(x)\rho_{2,x}\right)_{x}+b_2\rho_2  &=&\dis \alpha_1\psi_2^1\chi_{\omega_{1,d}}+\alpha_2\psi_2^2\chi_{\omega_{2,d}}& \mbox{in}& Q,\\
		\dis \rho_1(t,0)=\rho_1(t,1)=\rho_2(t,0)=\rho_2(t,1)&=&0& \mbox{on}& (0,T), \\
		\dis \rho_1(T,\cdot)=\rho_1^T,\ \ \rho_2(T,\cdot)=\rho_2^T&& &\mbox{in}&\Omega
	\end{array}
	\right.
\end{equation}
and 
\begin{equation}\label{psi}
	\left\{
	\begin{array}{rllll}
		\dis \psi_{1,t}^i-\left(a(x)\psi^i_{1,x}\right)_{x}+c_1\psi_1^i   &=&\dis-\frac{1}{\mu_i}\rho_{*}^{-2}\rho_1\chi_{\omega_{i}}& \mbox{in}& Q,\\
		\dis \psi_{2,t}^i-\left(a(x)\psi^i_{2,x}\right)_{x}+c_2\psi_2^i+d\psi_1^i   &=&0& \mbox{in}& Q,\\
		\dis \psi_1^i(t,0)=\psi_1^i(t,1)=\psi_2^i(t,0)=\psi_2^i(t,1)&=&0& \mbox{on}& (0,T), \\
		\dis \psi_1^i(0,\cdot)= \psi_2^i(0,\cdot)&=&0 &\mbox{in}&\Omega,
	\end{array}
	\right.
\end{equation}
 where $b_1,b_2,c_1,c_2,d\in L^\infty(Q)$ and $\rho^T=(\rho_1^T, \rho_2^T)\in [L^2(\Omega)]^2$.

Since the first assumption of \eqref{od} holds and if we set $\varrho_j=\alpha_1\psi_j^1+\alpha_2\psi_j^2,\ j=1,2$, then we obtain
\begin{equation}\label{rhobis}
	\left\{
	\begin{array}{rllll}
		\dis -\rho_{1,t}-\left(a(x)\rho_{1,x}\right)_{x}+b_1\rho_1+d\rho_2  &=&\dis \varrho_1\chi_{\omega_{d}}& \mbox{in}& Q,\\
		\dis -\rho_{2,t}-\left(a(x)\rho_{2,x}\right)_{x}+b_2\rho_2  &=&\dis \varrho_2\chi_{\omega_{d}}& \mbox{in}& Q,\\
		\dis \rho_1(t,0)=\rho_1(t,1)=\rho_2(t,0)=\rho_2(t,1)&=&0& \mbox{on}& (0,T), \\
		\dis \rho_1(T,\cdot)=\rho_1^T,\ \ \rho_2(T,\cdot)=\rho_2^T&& &\mbox{in}&\Omega
	\end{array}
	\right.
\end{equation}
and 
\begin{equation}\label{varrho}
	\left\{
	\begin{array}{rllll}
		\dis \varrho_{1,t}-\left(a(x)\varrho_{1,x}\right)_{x}+c_1\varrho_1   &=&\dis-\rho_{*}^{-2}\left(\frac{\alpha_1}{\mu_1}\chi_{\omega_{1}}+\frac{\alpha_2}{\mu_2}\chi_{\omega_{2}}\right)\rho_1& \mbox{in}& Q,\\
		\dis \varrho_{2,t}-\left(a(x)\varrho_{2,x}\right)_{x}+c_2\varrho_2+d\varrho_1   &=&0& \mbox{in}& Q,\\
		\dis \varrho_1(t,0)=\varrho_1(t,1)=\varrho_2(t,0)=\varrho_2(t,1)&=&0& \mbox{on}& (0,T), \\
		\dis \varrho_1(0,\cdot)= \varrho_2(0,\cdot)&=&0 &\mbox{in}&\Omega.
	\end{array}
	\right.
\end{equation}

Classically, to establish Carleman inequality, we define some weight functions according to the nature of model. In our case, these functions are stated as follows:\\
Since $\omega_d\cap \omega\neq \emptyset$ (recall \eqref{od}), then, there exists a non-empty open set $\O_1\Subset \omega_d\cap \omega$ and $\sigma\in \mathcal{C}^2([0,1])$ be a function such that
\begin{equation}\label{sigma}
\left\{
\begin{array}{llll}
 \sigma(x)>0\quad\text{in}\quad (0,1),\quad \sigma(0)=\sigma(1)=0,\\
\sigma_x(x)\neq 0\quad \text{in}\quad [0,1]\setminus\O_0,
\end{array}
\right.
\end{equation}
where $\O_0\Subset\O_1$ is an open subset. We refer to \cite{FursikovImanuvilov}  for the existence of such a function $\sigma$. \\
Let  $\tau\in [0,1)$ be as in the assumption \eqref{k} and  $r, \bar{d}\in \R$ be such that 
\begin{equation}\label{condrd}
r\geq \frac{4ln(2)}{\|\sigma\|_{\infty}} \hbox{ and } \bar{d}\geq \frac{5}{a(1)(2-\tau)}.
\end{equation}
If $r$ and $d$ verify \eqref{condrd}, then the interval $\dis I= \left[\frac{a(1)(2-\tau)(e^{2r\|\sigma\|_{\infty}}-1)}{\bar{d}\ a(1)(2-\tau)-1},
\frac{4(e^{2r\|\sigma\|_{\infty}}-e^{r\|\sigma\|_{\infty}})}{3\bar{d}}\right]$
is non-empty (see  \cite{birba2016}). Then we can choose $\lambda$ in this interval and for $r,\ \bar{d}$ satisfying \eqref{condrd}, we define the following functions:
\begin{equation}\label{functcarl}
\left\{
\begin{array}{llll}
\dis \Theta(t)=\frac{1}{(t(T-t))^4},\quad \forall t\in (0,T),\ \ \ \delta(x):=\dis \lambda\left(\int_{0}^{x}\frac{y}{a(y)}\ dy-\bar{d}\right),\\
\\
\dis\varphi(t,x):=\Theta(t)\delta(x),\quad \eta(t,x):=\Theta(t)e^{r\sigma(x)},\\
\\
\Psi(x)=\left(e^{r\sigma(x)}-e^{2r\|\sigma\|_\infty}\right),\ \ \ \Phi(t,x):=\Theta(t)\Psi(x).\\
\end{array}
\right.
\end{equation}
Using the second assumption in \eqref{condrd} on $\bar{d}$, we observe that $\delta(x)<0$ for all $x\in [0,1]$. Moreover, we have that $\Theta(t)\to +\infty$ as $t$ tends to $0^+$ and $T^-$.  Under assumptions \eqref{condrd} and the choice of the parameter $\lambda$, the weight functions $\varphi$ and $\Phi$ defined by \eqref{functcarl} satisfy the following inequalities which are needed in the sequel:
	
	\begin{equation}\label{ineqphi}
		\left\{
		\begin{array}{llll}
			\dis \frac{4}{3}\Phi\leq \varphi\leq\Phi\ \ \mbox{on}\ Q,\\
			\dis 2\Phi\leq \varphi\ \ \mbox{on}\ Q.
		\end{array}
		\right.
	\end{equation}	

\begin{remark}\label{r}$ $
	
	The weight function $\rho_{*}\in C^\infty([0,T])$ in \eqref{all16} is such that	
	\begin{equation}\label{wei}
	\rho_{*}(t)\geq e^{-s\varphi_{*}/2},	
	\end{equation}
 where $\dis \varphi_{*}=\min_{x\in \Omega}\varphi(t,x)$ and $\varphi$ is defined in \eqref{functcarl}. This  weight function will help us to prove a Carleman inequality for the solution of system \eqref{rho}-\eqref{psi}. We can refer to \cite{liliana2020, araruna2015anash, teresacorrig} for a similar use of weight function as \eqref{wei}.
\end{remark}

Before going further, we consider the following result inspired by \cite[Lemma 4.1]{maniar2011} useful for the rest of the paper.

\begin{leme}(Caccioppoli's inequality)\label{cacc}$ $
	
	Let $\O^\prime$ be a subset of $\O_1$ such that $\O^\prime\Subset\O_1$. Let $\rho=(\rho_1,\rho_2)$ and $\varrho=(\varrho_1,\varrho_2)$ be the solution of \eqref{rhobis} and \eqref{varrho} respectively. Then, there exists a positive constant $C$ such that
	\begin{equation}\label{caccio}
		\int_{0}^{T}\int_{\O^\prime}\sum_{i=1}^{2}(\rho^2_{i,x}+\varrho_{i,x}^2)\ e^{2s\varphi}\,\dq \leq C\int_0^T\int_{\O_1}s^2\Theta^2\sum_{i=1}^{2}(\rho_i^2+\varrho_i^2)\ e^{2s\varphi}\ \dq,
	\end{equation}	
	where the weight functions $\varphi$ and $\Theta$ are defined by \eqref{functcarl}.
\end{leme}

We state the following Carleman type inequality in the degenerate case proved in \cite[Proposition 3.3]{birba2016}.

\begin{prop}\label{prp1}$ $
	
	Consider the following system with $f\in L^2(Q)$, $a_0\in L^\infty(Q)$ and $z_T\in L^2(\Omega)$,
	
	\begin{equation}\label{defz1}
		\left\{
		\begin{array}{rllll}
			\dis -z_t-(a(x)z_x)_x+a_0z&=&f &\mbox{in}&Q,\\
			z(t,0)=z(t,1)&=&0  &\mbox{on}& (0,T),\\
			z(T,\cdot) &=&z_T  &\mbox{in}&  \Omega.
		\end{array}
		\right.
	\end{equation}
	Then, there exist two positive constants $C$ and $s_0$, such that every solution of \eqref{defz1} satisfies, for all $s\geq s_0$, the following inequality:
	
	\begin{eqnarray}\label{ineqcarl1}
		\dis\mathcal{I}(z)\leq C\left(\int_{Q}|f|^2e^{2s\varphi}\,\dq
		\dis +sa(1)\int_{0}^{T}\Theta z_x^2(t,1)e^{2s\varphi(t,1)}\,dt\right),
	\end{eqnarray}	
	where 
	\begin{equation}\label{I}
		\mathcal{I}(z)=\int_{Q_T}\frac{1}{s\Theta}\left(\left|z_t\right|^2+|(a(x)z_x)_x|^2\right)e^{2s\varphi}\  dxdt+\int_{Q}\left(s^3\Theta^3\frac{x^2}{a(x)}z^2+s\Theta a(x)z_x^2\right)e^{2s\varphi}\, \dq,	
	\end{equation}
	the functions $\Theta$ and $\varphi$ are given by \eqref{functcarl}.
\end{prop}

The next result is concerned with a classical Carleman estimate in suitable interval $(b_1,b_2)\subset [0,1]$ proved in \cite{FursikovImanuvilov}.

\begin{prop}\label{propcarl2}$ $
	
	We consider the following system with  $f\in L^2(Q_b)$, $a_0\in L^\infty(Q_b)$ and $a\in C^1([b_1;b_2])$ is a strictly positive function,
	\begin{equation}\label{carl4}
		\left\{
		\begin{array}{rllll}
			\dis -z_t-(a(x)z_x)_x+a_0z &=&f &\mbox{in}&Q_b,\\
			z(t,b_1)=z(t,b_2)&=&0  &\mbox{on}& (0,T),\\
		\end{array}
		\right.
	\end{equation}
	where $Q_b:=(0,T)\times (b_1,b_2)$.
	Then, there exist two positive constants $C$ and $s_2$, such that every solution of \eqref{carl4} satisfies, for all $s\geq s_1$, the following inequality holds
	
	\begin{eqnarray}\label{ineqcarl2}
		\mathcal{K}(z) \leq C\left(\int_{Q_b}|f|^2e^{2s\Phi}\,\dq
		+\int_{0}^{T}\int_{\O_1}s^3\eta^3z^2e^{2s\Phi}\, \dq\right),
	\end{eqnarray}
	where
	\begin{equation}\label{kc}
		\mathcal{K}(z)=\int_{Q_T}\frac{1}{s\eta}\left(\left|z_t\right|^2+|(a(x)z_x)_x|^2\right)e^{2s\Phi}\  dxdt+\int_{Q}(s^3\eta^3z^2+s\eta z_x^2)e^{2s\Phi}\,\dq
	\end{equation}
	and the functions $\eta$ and $\Phi$ are defined by \eqref{functcarl}.
\end{prop}

\subsection{An intermediate Carleman estimate}

Now, we state and prove one of the important result of this paper which is the intermediate Carleman estimate satisfied by solution of systems \eqref{rhobis}-\eqref{varrho}. This inequality is obtained by using the Carleman estimates \eqref{ineqcarl1} and \eqref{ineqcarl2}, the Hardy-Poincar\'e inequality \eqref{hardy} and the Caccioppoli's inequality \eqref{caccio}.

\begin{theorem}\label{thm1}$ $
	
	Assume that the hypotheses \eqref{k} on $a(\cdot)$ are satisfied. Then, there exists a  constant $C_1>0$ such that every solution $\rho$ and $\varrho$ of \eqref{rhobis} and \eqref{varrho} respectively, satisfy, for any $s$ large enough, the following inequality
	
	\begin{eqnarray}\label{ineqcarlprinc}
	\mathcal{I}(\rho_1)+\mathcal{I}(\rho_2)+\mathcal{I}(\varrho_1)+\mathcal{I}(\varrho_2) \leq C_1\int_{0}^{T}\int_{\O_1}s^3\Theta^3(\rho_1^2+\rho_2^2+\varrho_1^2+\varrho_2^2)e^{2s\Phi}\,\dq,
	\end{eqnarray}
  where the notation $\mathcal{I}(\cdot)$ is defined by \eqref{I}.
\end{theorem}

\begin{proof}
Let us choose an arbitrary open subset $\O^\prime:=(\alpha,\beta)$ such that $\O^\prime\Subset\O_1$. Let us introduce the smooth cut-off function $\xi:\R\to\R$ defined as follows
	\begin{equation}\label{cutoff}
	\left\{
	\begin{array}{llll}
	\dis 0\leq \xi\leq 1, \quad x\in \R,\\
	\dis \xi(x)=1, \quad x\in [0,\alpha],\\
	\dis \xi(x)=0, \quad x\in [\beta,1].
	\end{array}
	\right.
	\end{equation}	
	Let $\rho=(\rho_1,\rho_2)^{t}$ and $\varrho=(\varrho_1,\varrho_2)^{t}$ be respectively solutions of \eqref{rhobis} and \eqref{varrho}. We set $\widetilde{\rho}_i=\xi \rho_i,\ i=1,2$ and $\widetilde{\varrho}_i=\xi \varrho_i,\ i=1,2$. Then, $\widetilde{\rho}$ and $\widetilde{\varrho}$  are respectively solutions to
\begin{equation}\label{rhobist}
	\left\{
	\begin{array}{rllll}
		\dis -\widetilde{\rho}_{1,t}-\left(a(x)\widetilde{\rho}_{1,x}\right)_{x}+b_1\widetilde{\rho}_1+d\widetilde{\rho}_2  &=&\dis \widetilde{\varrho}_1\chi_{\omega_d}-\left(a(x)\xi_x\ \rho_1\right)_{x}-\xi_x\  a(x)\rho_{1,x}& \mbox{in}& Q,\\
		\dis -\widetilde{\rho}_{2,t}-\left(a(x)\widetilde{\rho}_{2,x}\right)_{x}+b_2\widetilde{\rho}_2  &=&\dis \widetilde{\varrho}_2\chi_{\omega_d}-\left(a(x)\xi_x\ \rho_2\right)_{x}-\xi_x\  a(x)\rho_{2,x}& \mbox{in}& Q,\\
		\dis \widetilde{\rho}_1(t,0)=\widetilde{\rho}_1(t,1)=\widetilde{\rho}_2(t,0)=\widetilde{\rho}_2(t,1)&=&0& \mbox{on}& (0,T), \\
		\dis \widetilde{\rho}_1(T,\cdot)=\widetilde{\rho}_1^T,\ \ \widetilde{\rho}_2(T,\cdot)=\widetilde{\rho}_2^T&& &\mbox{in}&\Omega
	\end{array}
	\right.
\end{equation}
and 
\begin{equation}\label{varrhot}
	\left\{
	\begin{array}{rllll}
		\dis \widetilde{\varrho}_{1,t}-\left(a(x)\widetilde{\varrho}_{1,x}\right)_{x}+c_1\widetilde{\varrho}_1   &=&\dis\widetilde{\mathcal{G}}& \mbox{in}& Q,\\
		\dis \widetilde{\varrho}_{2,t}-\left(a(x)\widetilde{\varrho}_{2,x}\right)_{x}+c_2\widetilde{\varrho}_2+d\widetilde{\varrho}_1   &=&-\left(a(x)\xi_x\ \varrho_2\right)_{x}-\xi_x\  a(x)\varrho_{2,x}& \mbox{in}& Q,\\
		\dis \widetilde{\varrho}_1(t,0)=\widetilde{\varrho}_1(t,1)=\widetilde{\varrho}_2(t,0)=\widetilde{\varrho}_2(t,1)&=&0& \mbox{on}& (0,T), \\
		\dis \widetilde{\varrho}_1(0,\cdot)= \widetilde{\varrho}_2(0,\cdot)&=&0 &\mbox{in}&\Omega,
	\end{array}
	\right.
\end{equation}
where $\dis \widetilde{\mathcal{G}}=-\rho_{*}^{-2}\left(\frac{\alpha_1}{\mu_1}\chi_{\omega_{1}}+\frac{\alpha_2}{\mu_2}\chi_{\omega_{2}}\right)\widetilde{\rho}_1-\left(a(x)\xi_x\ \varrho_1\right)_{x}-\xi_x\  a(x)\varrho_{1,x}$.

Let us also set $\overline{\rho}_i=\vartheta\rho_i,\ i=1,2$ and $\overline{\varrho}_i= \vartheta\varrho_i,\ i=1,2$ with $\vartheta=1-\xi$. Then, the support of $\overline{\rho}_i$ and $\overline{\varrho}_i$ is contained in $[0,T]\times[\alpha,1]$ and are respectively solutions to
\begin{equation}\label{rhobisb}
	\left\{
	\begin{array}{rllll}
		\dis -\overline{\rho}_{1,t}-\left(a(x)\overline{\rho}_{1,x}\right)_{x}+b_1\overline{\rho}_1+d\overline{\rho}_2  &=&\dis \overline{\varrho}_1\chi_{\O_d}-\left(a(x)\vartheta_x\ \rho_1\right)_{x}-\vartheta_x\  a(x)\rho_{1,x}& \mbox{in}& Q_{\alpha},\\
		\dis -\overline{\rho}_{2,t}-\left(a(x)\overline{\rho}_{2,x}\right)_{x}+b_2\overline{\rho}_2  &=&\dis \overline{\varrho}_2\chi_{\O_d}-\left(a(x)\vartheta_x\ \rho_2\right)_{x}-\vartheta_x\  a(x)\rho_{2,x}& \mbox{in}& Q_{\alpha},\\
		\dis \overline{\rho}_1(t,0)=\overline{\rho}_1(t,1)=\overline{\rho}_2(t,0)=\overline{\rho}_2(t,1)&=&0& \mbox{on}& (0,T), \\
		\dis \overline{\rho}_1(T,\cdot)=\overline{\rho}_1^T,\ \ \overline{\rho}_2(T,\cdot)=\overline{\rho}_2^T&& &\mbox{in}&\Omega
	\end{array}
	\right.
\end{equation}
and 
\begin{equation}\label{varrhob}
	\left\{
	\begin{array}{rllll}
		\dis \overline{\varrho}_{1,t}-\left(a(x)\overline{\varrho}_{1,x}\right)_{x}+c_1\overline{\varrho}_1   &=&\dis\overline{\mathcal{G}}& \mbox{in}& Q_{\alpha},\\
		\dis \widetilde{\varrho}_{2,t}-\left(a(x)\widetilde{\varrho}_{2,x}\right)_{x}+c_2\widetilde{\varrho}_2+d\widetilde{\varrho}_1   &=&-\left(a(x)\vartheta_x\ \varrho_2\right)_{x}-\vartheta_x\  a(x)\varrho_{2,x}& \mbox{in}& Q_{\alpha},\\
		\dis \overline{\varrho}_1(t,0)=\overline{\varrho}_1(t,1)=\overline{\varrho}_2(t,0)=\overline{\varrho}_2(t,1)&=&0& \mbox{on}& (0,T), \\
		\dis \overline{\varrho}_1(0,\cdot)= \overline{\varrho}_2(0,\cdot)&=&0 &\mbox{in}&\Omega,
	\end{array}
	\right.
\end{equation}
where, $\dis \overline{\mathcal{G}}=-\rho_{*}^{-2}\left(\frac{\alpha_1}{\mu_1}\chi_{\omega_{1}}+\frac{\alpha_2}{\mu_2}\chi_{\omega_{2}}\right)\overline{\rho}_1-\left(a(x)\vartheta_x\ \varrho_1\right)_{x}-\vartheta_x\  a(x)\varrho_{1,x}$ and  $Q_{\alpha}=(0,T)\times(\alpha,1)$. \\
We make the rest of the proof in three steps:

\noindent \textbf{Step 1.}
We prove that there exists a constant $C>0$ such that 
	\begin{eqnarray}\label{i5}
	\dis \mathcal{I}(\widetilde{\rho}_1)+\mathcal{I}(\widetilde{\rho}_2)+\mathcal{I}(\widetilde{\varrho}_1)+\mathcal{I}(\widetilde{\varrho}_2) \leq C\int_0^T\int_{\O^\prime}s^2\Theta^2(\rho_1^2+\rho_2^2+\varrho_1^2+\varrho_2^2)e^{2s\varphi}\,\dq.
\end{eqnarray}

Applying the Carleman estimate \eqref{ineqcarl1} to the first equation of \eqref{rhobist} with the second term $f=\dis -d\widetilde{\rho}_2+ \widetilde{\varrho}_1\chi_{\O_d}-\left(a(x)\xi_x\ \rho_1\right)_{x}-\xi_x\  a(x)\rho_{1,x}$, and using the fact that  $\widetilde{\rho}_{1,x}(t,1)=0$, we obtain
	\begin{eqnarray}\label{ineqcarl3rho}
		\mathcal{I}(\widetilde{\rho}_1)
		&\leq& \dis C\int_{Q}\left|\dis -d\widetilde{\rho}_2+ \widetilde{\varrho}_1\chi_{\O_d}-\left(a(x)\xi_x\ \rho_1\right)_{x}-\xi_x\  a(x)\rho_{1,x}\right|^2e^{2s\varphi}\,\dq\\
		&\leq&C \dis \int_{Q}\left[|d\widetilde{\rho}_2|^2+|\widetilde{\varrho}_1\chi_{\O_d}|^2+|(a(x)\xi_x{\rho}_1)_x+a(x)\xi_x{\rho}_{1,x}|^2\right] e^{2s\varphi}\,\dq.\nonumber
	\end{eqnarray}
	From the definition of $\xi$, we have
	\begin{eqnarray}\label{ine}
	\int_{Q} |(a(x)\xi_x{\rho}_1)_x+a(x)\xi_x{\rho}_{1,x}|^2 e^{2s\varphi}\,\dq&=& \int_{Q} ((a(x)\xi_x)_x{\rho}_1+2a(x)\xi_x{\rho}_{1,x})^2 e^{2s\varphi}\,\dq\nonumber\\
	&\leq& \int_{Q}\left[2((a(x)\xi_x)_x)^2{\rho_1}^2+8(a(x)\xi_x)^2{\rho}_{1,x}^2\right] e^{2s\varphi}\,\dq\nonumber\\
	&\leq& C\int_0^T\int_{\O^\prime}(\rho_1^2+\rho_{1,x}^2)\ e^{2s\varphi}\,\dq.
	\end{eqnarray}
	In the other hand, using the fact that $\dis \frac{x^2}{a(x)}$ is non-decreasing on $(0,1]$, thanks to Hardy-Poincar\'e inequality \eqref{hardy} to the function $e^{s\varphi}\widetilde{\varrho}_1$ and using the definition of $\varphi$, we get
	\begin{eqnarray*}
		\int_{Q}|\widetilde{\varrho}_1|^2e^{2s\varphi}\, \dq
		&\leq& \frac{1}{a(1)}\int_{Q} \frac{a(x)}{x^2}\widetilde{\varrho}_1^2e^{2s\varphi}\, \dq\\
		&\leq& \frac{\overline{C}}{a(1)}\int_{Q}a(x)|(\widetilde{\varrho}_1\ e^{s\varphi})_x|^2\, \dq\\
		&\leq& C\left(\int_{Q}a(x)\widetilde{\varrho}_{1,x}^2e^{2s\varphi}\, \dq+\int_{Q}s^2\Theta^2\frac{x^2}{a(x)}\widetilde{\varrho}_1^2e^{2s\varphi}\ \dq\right).
	\end{eqnarray*}
	Using the fact that there exist a positive constant $M_1$ such that 
	\begin{equation}\label{theta}
		1\leq M_1\Theta\ \ \mbox{and}\ \ \Theta^2\leq M_1\Theta^3,	
	\end{equation}
	we obtain 
	\begin{eqnarray}\label{ine1}
	\int_{Q}|\widetilde{\varrho}_1|^2e^{2s\varphi}\, \dq\leq  C\left(\int_{Q}\Theta a(x)\widetilde{\varrho}_{1,x}^2e^{2s\varphi}\, \dq+\int_{Q}s^2\Theta^3\frac{x^2}{a(x)}\widetilde{\varrho}_1^2e^{2s\varphi}\ \dq\right).
	\end{eqnarray}
Proceeding as before, we obtain 
\begin{eqnarray}\label{ca}
	\int_{Q}|d\widetilde{\rho}_2|^2e^{2s\varphi}\, \dq\leq  C\left(\int_{Q}\Theta a(x)\widetilde{\rho}_{2,x}^2e^{2s\varphi}\, \dq+\int_{Q}s^2\Theta^3\frac{x^2}{a(x)}\widetilde{\rho}_2^2e^{2s\varphi}\ \dq\right).
\end{eqnarray}
	Combining \eqref{ineqcarl3rho}-\eqref{ca}, we obtain
	\begin{eqnarray}\label{i1}
	\mathcal{I}(\widetilde{\rho}_1)
	\leq  C\int_0^T\int_{\O^\prime}(\rho_1^2+\rho_{1,x}^2)e^{2s\varphi}\,\dq 
	+C\int_{Q}\left(\Theta a(x)\widetilde{\varrho}_{1,x}^2+s^2\Theta^3\frac{x^2}{a(x)}\widetilde{\varrho}_1^2\right)e^{2s\varphi}\ \dq\nonumber\\
	+C\int_{Q}\left(\Theta a(x)\widetilde{\rho}_{2,x}^2+s^2\Theta^3\frac{x^2}{a(x)}\widetilde{\rho}_2^2\right)e^{2s\varphi}\ \dq.
	\end{eqnarray}
	Arguing in the same way as in \eqref{i1} with $\widetilde{\rho}_2$, $\widetilde{\varrho}_1$ and $\widetilde{\varrho}_2$, we respectively obtain
	
	\begin{eqnarray}\label{i2}
		\mathcal{I}(\widetilde{\rho}_2)
		\leq  C\int_0^T\int_{\O^\prime}(\rho_2^2+\rho_{2,x}^2)e^{2s\varphi}\,\dq 
		+C\int_{Q}\left(\Theta a(x)\widetilde{\varrho}_{2,x}^2+s^2\Theta^3\frac{x^2}{a(x)}\widetilde{\varrho}_2^2\right)e^{2s\varphi}\ \dq,
	\end{eqnarray}

\begin{eqnarray}\label{i3}
	\mathcal{I}(\widetilde{\varrho}_1)
	\leq  C\int_0^T\int_{\O^\prime}(\varrho_1^2+\varrho_{1,x}^2)e^{2s\varphi}\,\dq 
	+C\int_{Q}\left(\Theta a(x)\widetilde{\rho}_{1,x}^2+s^2\Theta^3\frac{x^2}{a(x)}\widetilde{\rho}_1^2\right)e^{2s\varphi}\ \dq
\end{eqnarray}
and 
\begin{eqnarray}\label{i4}
	\mathcal{I}(\widetilde{\varrho}_2)
	\leq  C\int_0^T\int_{\O^\prime}(\varrho_2^2+\varrho_{2,x}^2)e^{2s\varphi}\,\dq 
	+C\int_{Q}\left(\Theta a(x)\widetilde{\varrho}_{1,x}^2+s^2\Theta^3\frac{x^2}{a(x)}\widetilde{\varrho}_1^2\right)e^{2s\varphi}\ \dq.
\end{eqnarray}
	
	Combining \eqref{i1}-\eqref{i4} and taking $s$ large enough, we obtain
		
	\begin{eqnarray*}
		\dis \mathcal{I}(\widetilde{\rho}_1)+\mathcal{I}(\widetilde{\rho}_2)+\mathcal{I}(\widetilde{\varrho}_1)+\mathcal{I}(\widetilde{\varrho}_2) \leq C\sum_{i=1}^{2}\int_0^T\int_{\O^\prime}(\rho_i^2+\varrho_i^2+\rho^2_{i,x}+\varrho^2_{i,x})e^{2s\varphi}\,\dq.
	\end{eqnarray*}
	Using Caccioppoli's inequality \eqref{caccio} in the left hand side of this latter estimate, we obtain \eqref{i5}.
	 \\
	\noindent \textbf{Step 2.}
	We prove that there exists a constant $C>0$ such that 
\begin{eqnarray}\label{cmd}
	\dis \mathcal{K}(\overline{\rho}_1)+\mathcal{K}(\overline{\rho}_2)+\mathcal{K}(\overline{\varrho}_1)+\mathcal{K}(\overline{\varrho}_2) \leq C\int_0^T\int_{\O_1}s^3\Theta^3(\rho_1^2+\rho_2^2+\varrho_1^2+\varrho_2^2)e^{2s\Phi}\,\dq,
\end{eqnarray}
 where the notation $\mathcal{K}(\cdot)$ is defined by \eqref{kc}.

Since on $Q_{\alpha}$ all the above systems are non degenerate, applying the classical Carleman inequality \eqref{ineqcarl2}
 to the first solution $\overline{\rho}_1$ of \eqref{rhobisb} with $b_1=\alpha$, $b_2=1$ and $f=\dis -d\overline{\rho}_2+ \overline{\varrho}_1\chi_{\O_d}-\left(a(x)\vartheta_x\ \rho_1\right)_{x}-\vartheta_x\  a(x)\rho_{1,x}$, we get
\begin{eqnarray}\label{cl}
	\mathcal{K}(\overline{\rho}_1)
	&\leq& C \dis \int_{Q}\left[|d\overline{\rho}_2|^2+|\overline{\varrho}_1\chi_{\omega_d}|^2+|(a(x)\vartheta_x{\rho}_1)_x+a(x)\vartheta_x{\rho}_{1,x}|^2\right] e^{2s\Phi}\,\dq+\int_0^T\int_{\O_1}s^3\eta^3\overline{\rho}_1^2e^{2s\Phi}\,\dq\nonumber\\
	&\leq& C \dis \int_{Q}\left[|d\overline{\rho}_2|^2+|\overline{\varrho}_1\chi_{\omega_d}|^2+|(a(x)\vartheta_x{\rho}_1)_x+a(x)\vartheta_x{\rho}_{1,x}|^2\right] e^{2s\Phi}\,\dq\\
	&+&\int_0^T\int_{\O_1}s^3\Theta^3\overline{\rho}_1^2e^{2s\Phi}\,\dq\nonumber
\end{eqnarray}
because $\dis\eta(t,x):=\Theta(t)e^{r\sigma(x)}\leq \Theta(t)e^{r\|\sigma(x)\|_{\infty}}$.
Using the definition of the function $\vartheta$, we have
\begin{eqnarray}\label{}
	\int_{Q} |(a(x)\vartheta_x{\rho}_1)_x+a(x)\vartheta_x{\rho}_{1,x}|^2 e^{2s\Phi}\,\dq&=& \int_{Q} ((a(x)\vartheta_x)_x{\rho}_1+2a(x)\vartheta_x{\rho}_{1,x})^2 e^{2s\Phi}\,\dq\nonumber\\
	&\leq& \int_{Q}\left[2((a(x)\vartheta_x)_x)^2{\rho_1}^2+8(a(x)\vartheta_x)^2{\rho}_{1,x}^2\right] e^{2s\Phi}\,\dq\nonumber\\
	&\leq& C\int_0^T\int_{\O^\prime}(\rho_1^2+\rho_{1,x}^2)\ e^{2s\Phi}\,\dq.
\end{eqnarray}
On the other hand, since $\dis \frac{x^2}{a(x)}$ is non-decreasing on $(0,1]$ and thanks to Hardy-Poincar\'e inequality \eqref{hardy} to the function $e^{s\Phi}\overline{\varrho}_1$, we get
\begin{eqnarray*}
	\int_{Q}|\overline{\varrho}_1|^2e^{2s\Phi}\, \dq
	&\leq& \frac{1}{a(1)}\int_{Q} \frac{a(x)}{x^2}\overline{\varrho}_1^2e^{2s\Phi}\, \dq\\
	&\leq& \frac{\overline{C}}{a(1)}\int_{Q}a(x)|(\overline{\varrho}_1\ e^{s\Phi})_x|^2\, \dq\\
&\leq& C\int_{Q}\left(a(x)\overline{\varrho}_{1,x}^2+a(x)s^2\eta^2\overline{\varrho}_1^2\right)e^{2s\Phi}\ \dq.
\end{eqnarray*}
Using \eqref{theta}, the fact that $a\in \mathcal{C}^1([\alpha;1])$ and $\eta^{-1}\in L^\infty(Q)$, we get
\begin{eqnarray}\label{i8}
\int_{Q}\overline{\varrho}_1^2e^{2s\Phi}\, \dq\leq  C\int_{Q}\left(\eta \overline{\varrho}_{1,x}^2+s^2\eta^3\overline{\varrho}_1^2\right)e^{2s\Phi}\ \dq.
\end{eqnarray}
Arguing as before, we obtain 
\begin{eqnarray}\label{i7}
	\int_{Q}|d\overline{\rho}_2|^2e^{2s\Phi}\, \dq\leq  C\int_{Q}\left(\eta \overline{\rho}_{2,x}^2+s^2\eta^3\overline{\rho}_2^2\right)e^{2s\Phi}\ \dq.
\end{eqnarray}

	Combining \eqref{cl}-\eqref{i7}, we obtain
	
	\begin{eqnarray}\label{cm}
	\mathcal{K}(\overline{\rho}_1)
	\leq  C\int_0^T\int_{\O^\prime}(\rho_1^2+\rho_{1,x}^2)e^{2s\Phi}\,\dq+C\int_0^T\int_{\O_1}s^3\Theta^3\rho_1^2 e^{2s\Phi}\,\dq\nonumber\\
	+C\int_{Q}\left(\eta \overline{\varrho}_{1,x}^2+s^2\eta^3\overline{\varrho}_1^2\right)e^{2s\Phi}\ \dq+C\int_{Q}\left(\eta \overline{\rho}_{2,x}^2+s^2\eta^3\overline{\rho}_2^2\right)e^{2s\Phi}\ \dq.
	\end{eqnarray}
	
Applying the same way as in \eqref{cm} to $\overline{\rho}_2$, $\overline{\varrho}_1$ and $\overline{\varrho}_2$, we respectively obtain
	
	\begin{eqnarray}\label{cma}
	\mathcal{K}(\overline{\rho}_2)
	\leq  C\int_0^T\int_{\O^\prime}(\rho_2^2+\rho_{2,x}^2)e^{2s\Phi}\,\dq+C\int_0^T\int_{\O_1}s^3\Theta^3\rho_2^2 e^{2s\Phi}\,\dq\nonumber\\
	+C\int_{Q}\left(\eta \overline{\varrho}_{2,x}^2+s^2\eta^3\overline{\varrho}_2^2\right)e^{2s\Phi}\ \dq,
\end{eqnarray}
	\begin{eqnarray}\label{cmb}
	\mathcal{K}(\overline{\varrho}_1)
	\leq  C\int_0^T\int_{\O^\prime}(\varrho_1^2+\varrho_{1,x}^2)e^{2s\Phi}\,\dq+C\int_0^T\int_{\O_1}s^3\Theta^3\varrho_1^2 e^{2s\Phi}\,\dq\nonumber\\
	+C\int_{Q}\left(\eta \overline{\rho}_{1,x}^2+s^2\eta^3\overline{\rho}_1^2\right)e^{2s\Phi}\ \dq
\end{eqnarray}
and 
	\begin{eqnarray}\label{cmc}
	\mathcal{K}(\overline{\varrho}_2)
	\leq  C\int_0^T\int_{\O^\prime}(\varrho_2^2+\varrho_{2,x}^2)e^{2s\Phi}\,\dq+C\int_0^T\int_{\O_1}s^3\Theta^3\varrho_2^2 e^{2s\Phi}\,\dq\nonumber\\
	+C\int_{Q}\left(\eta \overline{\varrho}_{1,x}^2+s^2\eta^3\overline{\varrho}_1^2\right)e^{2s\Phi}\ \dq.
\end{eqnarray}
	Combining \eqref{cm}-\eqref{cmc} and taking $s$ large enough, we obtain
	
	\begin{eqnarray*}
		\dis \mathcal{K}(\overline{\rho}_1)+\mathcal{K}(\overline{\rho}_2)+\mathcal{K}(\overline{\varrho}_1)+\mathcal{K}(\overline{\varrho}_2) \leq C\sum_{i=1}^{2}\int_0^T\int_{\O^\prime}(\rho_i^2+\varrho_i^2+\rho^2_{i,x}+\varrho^2_{i,x})e^{2s\Phi}\,\dq\nonumber\\
		+C\int_0^T\int_{\O_1}s^3\Theta^3(\rho_1^2+\rho_2^2+\varrho_1^2+\varrho_2^2)e^{2s\Phi}\,\dq.
	\end{eqnarray*}
	Combining this latter estimate with Caccioppoli's inequality \eqref{caccio}, we obtain we obtain \eqref{cmd}.
 \\
\noindent \textbf{Step 3.}
Now, we prove the inequality \eqref{ineqcarlprinc}.
	
Thanks to inequalities \eqref{ineqphi}, the fact that $a\in \mathcal{C}^1([\alpha,1])$ and the function $\dis \frac{x^2}{a(x)}$ is non-decreasing on $(0,1]$, one can prove the existence of a constant $C>0$  such that for all $(t,x)\in (0,T)\times [\alpha,1]$, we have
\begin{equation}\label{est}
\begin{array}{llll}
\dis e^{2s\varphi}\leq e^{2s\Phi},\ \ \ \, \frac{x^2}{a(x)}e^{2s\varphi}\leq Ce^{2s\Phi},\ \ \ \, a(x)e^{2s\varphi}\leq Ce^{2s\Phi}.
\end{array}
\end{equation}
Using \eqref{est}, the inequality \eqref{cmd} becomes
	
	\begin{eqnarray}\label{i17}
	\dis \mathcal{I}(\overline{\rho}_1)+\mathcal{I}(\overline{\rho}_2)+\mathcal{I}(\overline{\varrho}_1)+\mathcal{I}(\overline{\varrho}_2) \leq C\int_0^T\int_{\O_1}s^3\Theta^3(\rho_1^2+\rho_2^2+\varrho_1^2+\varrho_2^2)e^{2s\Phi}\,\dq.
\end{eqnarray}	
Combining \eqref{i5} and \eqref{i17} and using the first inequality of \eqref{est}, we obtain	
	\begin{eqnarray}\label{i18}
	\dis \mathcal{I}(\overline{\rho}_1+\widetilde{\rho}_1)+\mathcal{I}(\overline{\rho}_2+\widetilde{\rho}_2)+\mathcal{I}(\overline{\varrho}_1+\widetilde{\varrho}_1)+\mathcal{I}(\overline{\varrho}_2+\widetilde{\varrho}_2) \leq C\int_0^T\int_{\O_1}s^3\Theta^3(\rho_1^2+\rho_2^2+\varrho_1^2+\varrho_2^2)e^{2s\Phi}\,\dq.
\end{eqnarray}		
Using the fact that $\varrho_i=\widetilde{\varrho}_i+\overline{\varrho}_i,\ i=1,2$ and $\rho_i=\widetilde{\rho}_i+\overline{\rho}_i,\ i=1,2$, then we have 
\begin{equation}\label{pat}
	\begin{array}{rll}
	\dis	|\varrho_i|^2\leq 2\left(|\widetilde{\varrho}_i|^2+|\overline{\varrho}_i|^2\right),\ |\rho_i|^2\leq 2\left(|\widetilde{\rho}_i|^2+|\overline{\rho}_i|^2\right),\\
		\\
		 \dis |\varrho_{i,x}|^2\leq 2\left(|\widetilde{\varrho}_{i,x}|^2+|\overline{\varrho}_{i,x}|^2\right),\ |\rho_{i,x}|^2\leq 2\left(|\widetilde{\rho}_{i,x}|^2+|\overline{\rho}_{i,x}|^2\right). 
	\end{array}
\end{equation}
Combining \eqref{i18} and \eqref{pat}, we obtain the existence of a constant $C_1>0$ such that	
\begin{eqnarray*}
	\mathcal{I}(\rho_1)+\mathcal{I}(\rho_2)+\mathcal{I}(\varrho_1)+\mathcal{I}(\varrho_2) \leq C_1\int_{0}^{T}\int_{\O_1}s^3\Theta^3(\rho_1^2+\rho_2^2+\varrho_1^2+\varrho_2^2)e^{2s\Phi}\,\dq.
\end{eqnarray*}
This completes the proof.
\end{proof}

\subsection{An observability inequality result}

This part is devoted to the observability inequality of systems \eqref{rhobis}-\eqref{varrho}. This inequality is obtained by using the intermediate Carleman estimate \eqref{ineqcarlprinc}.

\begin{prop} \label{prop5}  $ $
	
	Under the assumptions of Theorem \ref{thm1}, there exists a constant $C_4>0$, such that every solution $\rho$ and $\varrho$ of \eqref{rhobis} and \eqref{varrho}, respectively, satisfy, for $s$ large enough, the following inequality:
		
	\begin{eqnarray}\label{obser}
		\mathcal{I}(\rho_1)+\mathcal{I}(\rho_2)+\mathcal{I}(\varrho_1)+\mathcal{I}(\varrho_2) \leq C_4s^{15}\int_{0}^{T}\int_{\omega}|\rho_1|^2\,\dq,
	\end{eqnarray}
	where the notation $\mathcal{I}(\cdot)$ is defined by \eqref{I}.
	
\end{prop}

\begin{proof}
	
The proof of this proposition is inspired by techniques from \cite[Lemma 1]{teresa2000insensitizing}.	We will achieve it in two steps.\\	
	\textbf{Step 1.}	
	First, we want to eliminate the local term corresponding to $\varrho_1$ and $\varrho_2$ on the right hand side of \eqref{ineqcarlprinc}.
	\noindent So, let $\O_2$ be a non empty open set such that $\O_1\Subset \O_2\Subset \omega_d\cap\omega$. We introduce as in \cite{teresa2000insensitizing} the cut off function $\xi_1\in C^{\infty}_0(\Omega)$ such that \begin{subequations}\label{owogene1}
		\begin{alignat}{11}
		\dis 	0\leq \xi_1\leq 1\ \mbox{in}\ \Omega, \,\,\xi_1=1 \hbox{ in } \O_1 ,\,\,  \xi_1=0 \hbox{ in } \Omega\setminus\O_2,\\
		\dis 	\frac{\xi_{1,xx}}{\xi_1^{1/2}}\in L^{\infty}(\O_2),\,\, \frac{\xi_{1,x}}{\xi_1^{1/2}}\in L^{\infty}(\O_2).
		\end{alignat}
	\end{subequations}
Set $\dis u_1=s^3\Theta^3e^{2s\Phi}$. Then $u_1(T)=u_1(0)=0$ and we have the following estimations:

\begin{equation}\label{con}
	\begin{array}{rll}
		\dis |u_1\xi_1|\leq s^3\Theta^3e^{2s\Phi}\xi_1,\ \ \ \ \ \
		\dis \left|(u_1\xi_1)_t\right|\leq Cs^4\Theta^8e^{2s\Phi}\xi_1,\\
		\\
		\dis |(u_1\xi_1)_x|\leq Cs^4\Theta^4e^{2s\Phi}\xi_1,\ \ \ \ \ \
		\dis |(a(x)(u_1\xi_1)_x)_x|\leq Cs^5\Theta^5e^{2s\Phi}\xi_1,
	\end{array}
\end{equation}
where $C$ is a positive constant.\par 
	If we multiply the first and the second  equation of \eqref{rhobis} by $u_1\xi_1\varrho_1$ and $u_1\xi_1\varrho_2$, respectively, and integrate by parts over $Q$, we obtain
	\begin{equation}\label{beau}
		J_1+J_2+J_3+J_4+J_5+J_6+J_7=\int_{Q}u_1\xi_1(|\varrho_1|^2+|\varrho_2|^2)\chi_{\omega_d}\ \dT,
	\end{equation}
where
	$$\begin{array}{lllll}
	J_1=\dis-\frac{\alpha_1}{\mu_1}\int_{Q}u_1\xi_1|\rho_1|^2\chi_{\omega_{1}} \ \dT-\frac{\alpha_2}{\mu_2}\int_{Q}u_1\xi_1|\rho_1|^2\chi_{\omega_{2}} \ \dT,\,
	J_2=\int_{Q}\rho_1\varrho_1\frac{\partial (u_1\xi_1)}{\partial t}\ \dT,\\
	\dis J_3=\int_{Q}\rho_2\varrho_2\frac{\partial (u_1\xi_1)}{\partial t}\ \dT,\  J_4= \dis -\int_{Q}(a(x)(u_1\xi_1)_x)_x\ \rho_1\varrho_1\ \dT,\ J_5= \dis -\int_{Q}(a(x)(u_1\xi_1)_x)_x\ \rho_2\varrho_2\ \dT,\\
	\dis J_6=\dis-2\int_Qa(x)(u_1\xi_1)_x\rho_1\varrho_{1,x}\ \dT,\ J_7=\dis-2\int_Qa(x)(u_1\xi_1)_x\rho_2\varrho_{2,x}\ \dT.
	\end{array}
	$$
		
	Let us estimate $J_i,\ i=1,\cdots,7$. Using Young's inequality, we have
	
	\begin{eqnarray*}
		J_1&\leq& \dis \left(\frac{\alpha_1}{\mu_1}+\frac{\alpha_2}{\mu_2}\right)C\int_{Q} s^3\Theta^3e^{2s\Phi}\xi_1|\rho_1|^2\ \dT\\
		&\leq&\dis \left(\frac{\alpha_1^2}{\mu_1^2}+\frac{\alpha_2^2}{\mu_2^2}\right)\int_Q s^3\Theta^3\frac{x^2}{a(x)}e^{2s\varphi}|\rho_1|^2\dT+C\int_{0}^{T}\int_{\O_2} s^3\Theta^3\frac{a(x)}{x^2}e^{2s(2\Phi-\varphi)}|\rho_1|^2\dT,	
	\end{eqnarray*}
	
	\begin{eqnarray*}
		J_2&\leq&\dis C\int_{Q}s^4\Theta^8e^{2s\Phi}\xi_1|\rho_1\varrho_1|\ \dT\\
		&\leq&\dis \frac{\delta_2}{2}\int_Q s^3\Theta^3\frac{x^2}{a(x)}e^{2s\varphi}|\varrho_1|^2\dT+C_{\delta_2}\int_{0}^{T}\int_{\O_2} s^5\Theta^{13}\frac{a(x)}{x^2}e^{2s(2\Phi-\varphi)}|\rho_1|^2\dT,
	\end{eqnarray*}
	
	\begin{eqnarray*}
		J_3&\leq&\dis C\int_{Q}s^4\Theta^8e^{2s\Phi}\xi_1|\rho_2\varrho_2|\ \dT\\
		&\leq&\dis \frac{\delta_3}{2}\int_Q s^3\Theta^3\frac{x^2}{a(x)}e^{2s\varphi}|\varrho_2|^2\dT+C_{\delta_3}\int_{0}^{T}\int_{\O_2} s^5\Theta^{13}\frac{a(x)}{x^2}e^{2s(2\Phi-\varphi)}|\rho_2|^2\dT,
	\end{eqnarray*}
	
	\begin{eqnarray*}
		J_4&\leq&C\int_{Q}s^5\Theta^5e^{2s\Phi}\xi_1|\rho_1\varrho_1|\ \dT\\
		&\leq&\dis \frac{\delta_4}{2}\int_Q s^3\Theta^3\frac{x^2}{a(x)}e^{2s\varphi}|\varrho_1|^2\dT+C_{\delta_4}\int_{0}^{T}\int_{\O_2} s^7\Theta^7\frac{a(x)}{x^2}e^{2s(2\Phi-\varphi)}|\rho_1|^2\dT,
	\end{eqnarray*}

\begin{eqnarray*}
	J_5&\leq&C\int_{Q}s^5\Theta^5e^{2s\Phi}\xi_1|\rho_2\varrho_2|\ \dT\\
	&\leq&\dis \frac{\delta_5}{2}\int_Q s^3\Theta^3\frac{x^2}{a(x)}e^{2s\varphi}|\varrho_2|^2\dT+C_{\delta_5}\int_{0}^{T}\int_{\O_2} s^7\Theta^7\frac{a(x)}{x^2}e^{2s(2\Phi-\varphi)}|\rho_2|^2\dT,
\end{eqnarray*}
	
		\begin{eqnarray*}
			J_6&\leq&C\int_{Q}s^4\Theta^4a(x)e^{2s\Phi}\xi_1|\rho_1\varrho_{1,x}|\ \dT\\
			&\leq&\dis \frac{\delta_6}{2}\int_Q s\Theta a(x)e^{2s\varphi}|\varrho_{1,x}|^2\dT+C_{\delta_6}\int_{0}^{T}\int_{\O_2} s^7\Theta^7a(x)e^{2s(2\Phi-\varphi)}|\rho_1|^2\dT,
		\end{eqnarray*}
	
		\begin{eqnarray*}
		J_7&\leq&C\int_{Q}s^4\Theta^4a(x)e^{2s\Phi}\xi_1|\rho_2\varrho_{2,x}|\ \dT\\
		&\leq&\dis \frac{\delta_7}{2}\int_Q s\Theta a(x)e^{2s\varphi}|\varrho_{2,x}|^2\dT+C_{\delta_7}\int_{0}^{T}\int_{\O_2} s^7\Theta^7a(x)e^{2s(2\Phi-\varphi)}|\rho_2|^2\dT.
	\end{eqnarray*}
	
	Finally, choosing the constants $\delta_i$ such that $\dis \delta_2=\delta_3=\delta_4=\delta_5=\frac{1}{4C_1}$ and $\dis \delta_6=\delta_7=\frac{1}{2C_1}$, where $C_1$ is the constant obtained to Theorem \ref{thm1}, it follows from \eqref{beau} and the previous inequalities that
	\begin{equation}\label{owo4}
	\begin{array}{rll}
	\dis \int_0^T\int_{\O_1}s^3\Theta^3e^{2s\Phi}(|\varrho_1|^2+|\varrho_2|^2)\ \dT	
	&\leq&\dis \frac{1}{2C_1}\mathcal{I}(\varrho_1)+\frac{1}{2C_1}\mathcal{I}(\varrho_2)\\
	&+&\dis \left(\frac{\alpha_1^2}{\mu_1^2}+\frac{\alpha_2^2}{\mu_2^2}\right)\int_0^T\int_{Q} s^3\Theta^3\frac{x^2}{a(x)}e^{2s\varphi}|\rho_1|^2\ \dT\\
	&+&\dis C\int_{0}^{T}\int_{\O_2} s^7\Theta^{13}\frac{a(x)}{x^2}e^{2s(2\Phi-\varphi)}(|\rho_1|^2+|\rho_2|^2)\ \dT\\
	&&\dis +C\int_{0}^{T}\int_{\O_2} s^7\Theta^7a(x)e^{2s(2\Phi-\varphi)}(|\rho_1|^2+|\rho_2|^2)\ \dT.
	\end{array}
	\end{equation}
	Combining \eqref{ineqcarlprinc} with \eqref{owo4} and taking $\mu_i,\ i=1,2$ large enough, we obtain 
	\begin{eqnarray}\label{obser2}
	\dis \mathcal{I}(\rho_1)+\mathcal{I}(\rho_2)+\mathcal{I}(\varrho_1)+\mathcal{I}(\varrho_2)\leq \dis C\int_{0}^{T}\int_{\O_2} s^7\Theta^{13}\frac{a(x)}{x^2}e^{2s(2\Phi-\varphi)}(|\rho_1|^2+|\rho_2|^2)\ \dT\nonumber\\
	\dis +C\int_{0}^{T}\int_{\O_2} s^7\Theta^{7}a(x)e^{2s(2\Phi-\varphi)}(|\rho_1|^2+|\rho_2|^2)\ \dT.
	\end{eqnarray}
Since $\dis \frac{a(x)}{x^2}$ and $a(x)$ are bounded on $\O_2$, 
then using \eqref{obser2}, we obtain the existence of a positive constant $C_2$ such that 	
		\begin{eqnarray}\label{obser1}
		\dis \mathcal{I}(\rho_1)+\mathcal{I}(\rho_2)+\mathcal{I}(\varrho_1)+\mathcal{I}(\varrho_2)
		\leq C_2\int_{0}^{T}\int_{\O_2} s^7\Theta^{13}e^{2s(2\Phi-\varphi)}(|\rho_1|^2+|\rho_2|^2)\ \dT.
		\end{eqnarray} 	
\textbf{Step 2.}	
Now, we want to eliminate the local term $\rho_2$ on the right hand side of \eqref{obser1}.
We take a non empty open set $\O_3$ such that $\O_2\Subset \O_3\Subset \omega_d\cap\omega$. Let $\xi_2\in C^{\infty}_0(\Omega)$ a cut off function  such that 
\begin{subequations}\label{owogene1a}
	\begin{alignat}{11}
		\dis 	0\leq \xi_2\leq 1\ \mbox{in}\ \Omega, \,\,\xi_2=1 \hbox{ in } \O_2 ,\,\,  \xi_2=0 \hbox{ in } \Omega\setminus\O_3,\\
		\dis 	\frac{\xi_{2,xx}}{\xi_2^{1/2}}\in L^{\infty}(\O_3),\,\, \frac{\xi_{2,x}}{\xi_2^{1/2}}\in L^{\infty}(\O_3).
	\end{alignat}
\end{subequations}
We set $\dis u_2=s^7\Theta^{13}e^{2s(2\Phi-\varphi)}$. Then, there exists a positive constant $C$ such that the following inequalities holds:

\begin{equation}\label{cona}
	\begin{array}{rll}
		\dis |u_2\xi_2|\leq s^7\Theta^{13}e^{2s(2\Phi-\varphi)}\xi_2,\ \ \ \ \ \
		\dis \left|(u_2\xi_2)_t\right|\leq Cs^8\Theta^{18}e^{2s(2\Phi-\varphi)}\xi_2,\\
		\\
		\dis |(u_2\xi_2)_x|\leq Cs^8\Theta^{14}e^{2s(2\Phi-\varphi)}\xi_2,\ \ \ \ \ \
		\dis |(a(x)(u_2\xi_2)_x)_x|\leq Cs^9\Theta^{15}e^{2s(2\Phi-\varphi)}\xi_2.
	\end{array}
\end{equation}
Multiplying the first  equation of \eqref{rhobis} by $u_2\xi_2\rho_2$ and integrating by parts over $Q$, we obtain
\begin{equation}\label{beaua}
	\int_{Q}du_2\xi_2|\rho_2|^2\ \dT=K_1+K_2+K_3+K_4+K_5+K_6+K_7,
\end{equation}
where
$$
\begin{array}{lllll}
	\dis K_1=-\int_{Q}\rho_1\rho_2(u_2\xi_2)_t\ \dT,\ \ K_2=-\int_{Q}u_2\xi_2\rho_1\rho_{2,t}\ \dT,\ K_3= \dis \int_{Q}(a(x)(u_2\xi_2)_x)_x\ \rho_1\rho_2\ \dT,\\
	\dis K_4=\dis2\int_Qa(x)(u_2\xi_2)_x\rho_1\rho_{2,x}\ \dT,\  K_5= \dis \int_{Q}u_2\xi_2\rho_1(a(x)\rho_{2,x})_x\ \dT,\\ K_6= \dis- \int_{Q}b_1^{w_1}u_2\xi_2\rho_1\rho_2\ \dT,\ K_7= \dis \int_{Q}u_2\xi_2\varrho_1\rho_2\chi_{\omega_d}\ \dT.
\end{array}
$$
Let us estimate $K_i,\ i=1,\cdots,7$. Using inequalities \eqref{cona}, we have

\begin{eqnarray*}
	K_1\leq\dis \frac{\gamma_1}{2}\int_Q s^3\Theta^3\frac{x^2}{a(x)}e^{2s\varphi}|\rho_2|^2\dT+C_{\gamma_1}\int_{0}^{T}\int_{\O_3} s^{13}\Theta^{33}\frac{a(x)}{x^2}e^{2s(4\Phi-3\varphi)}|\rho_1|^2\dT,
\end{eqnarray*}

\begin{eqnarray*}
	K_2\leq\dis \frac{\gamma_2}{2}\int_Q\frac{1}{s\Theta} e^{2s\varphi}|\rho_{2,t}|^2\dT+C_{\gamma_2}\int_{0}^{T}\int_{\O_3} s^{15}\Theta^{27}e^{2s(4\Phi-3\varphi)}|\rho_1|^2\dT,
\end{eqnarray*}

\begin{eqnarray*}
	K_3\leq\dis \frac{\gamma_3}{2}\int_Q s^3\Theta^3\frac{x^2}{a(x)}e^{2s\varphi}|\rho_2|^2\dT+C_{\gamma_3}\int_{0}^{T}\int_{\O_3} s^{15}\Theta^{27}\frac{a(x)}{x^2}e^{2s(4\Phi-3\varphi)}|\rho_1|^2\dT,
\end{eqnarray*}

\begin{eqnarray*}
	K_4\leq\dis \frac{\gamma_4}{2}\int_Q s\Theta a(x)e^{2s\varphi}|\rho_{2,x}|^2\dT+C_{\gamma_4}\int_{0}^{T}\int_{\O_3} s^{15}\Theta^{27}a(x)e^{2s(4\Phi-3\varphi)}|\rho_1|^2\dT,
\end{eqnarray*}	

\begin{eqnarray*}
	K_5\leq\dis \frac{\gamma_5}{2}\int_Q\frac{1}{s\Theta} e^{2s\varphi}|(a(x)\rho_{2,x})_x|^2\dT+C_{\gamma_5}\int_{0}^{T}\int_{\O_3} s^7\Theta^{7}e^{2s(4\Phi-3\varphi)}|\rho_1|^2\dT,
\end{eqnarray*}

\begin{eqnarray*}
	K_6\leq\dis \frac{\gamma_6}{2}\int_Q s^3\Theta^3\frac{x^2}{a(x)}e^{2s\varphi}|\rho_2|^2\dT+C_{\gamma_6}\int_{0}^{T}\int_{\O_3} s^{11}\Theta^{23}\frac{a(x)}{x^2}e^{2s(4\Phi-3\varphi)}|\rho_1|^2\dT.
\end{eqnarray*}
For $K_7$, we have
\begin{eqnarray*}
	K_7\leq\dis \frac{\gamma_7}{2}\int_Q s^3\Theta^3\frac{x^2}{a(x)}e^{2s\varphi}|\rho_2|^2\dT+C_{\gamma_7}\int_{0}^{T}\int_{\O_3} s^{11}\Theta^{23}\frac{a(x)}{x^2}e^{2s(4\Phi-3\varphi)}|\varrho_1|^2\dT.
\end{eqnarray*}
Thanks to \eqref{ineqphi}, we have $4\Phi-3\varphi\leq 0$, then $\dis \Theta^{23}e^{2s(4\Phi-3\varphi)}\in L^\infty(Q)$. Furthermore, $\dis \frac{a(x)}{x^2}$ is bounded on $\O_3$. Then, $K_7$ becomes
\begin{eqnarray*}
	K_7\leq\dis \frac{\gamma_7}{2}\int_Q s^3\Theta^3\frac{x^2}{a(x)}e^{2s\varphi}|\rho_2|^2\dT+C_{\gamma_7,s}\int_{0}^{T}\int_{\O_3} |\varrho_1|^2\dT.
\end{eqnarray*}

We choose the constants $\gamma_i$ such that $\dis \gamma_1=\gamma_3=\gamma_6=\gamma_7=\frac{d_0}{4C_2}$ and $\dis \gamma_2=\gamma_5=\gamma_4=\frac{d_0}{C_2}$, where $C_2$ is the constant obtained in \eqref{obser1} and $d_0$ is defined in \eqref{d}. Using the condition \eqref{d} and the fact that $\dis \frac{a(x)}{x^2}$ and $a(x)$ are bounded on $\O_3$, we obtain from \eqref{beaua} 
\begin{equation}\label{owo4a}
	\begin{array}{rll}
		\dis d_0\int_0^T\int_{\O_2}s^7\Theta^{13}e^{2s(4\Phi-3\varphi)}|\varrho_2|^2\ \dT	
		\leq\dis \frac{d_0}{2C_2}\mathcal{I}(\varrho_2)+\dis C_3\int_{0}^{T}\int_{\O_3} s^{15}\Theta^{33}e^{2s(4\Phi-3\varphi)}|\rho_1|^2\ \dT\\
		\dis +C(s)\int_{0}^{T}\int_{\O_3} |\varrho_1|^2\dT.
	\end{array}
\end{equation}
Putting the estimate \eqref{owo4a} in \eqref{obser1}, we get
	\begin{eqnarray}\label{obser1a}
	\dis \mathcal{I}(\rho_1)+\mathcal{I}(\rho_2)+\mathcal{I}(\varrho_1)+\mathcal{I}(\varrho_2)
	\leq C_4\int_{0}^{T}\int_{\O_3} s^{15}\Theta^{33}e^{2s(4\Phi-3\varphi)}|\rho_1|^2\ \dT +C(s)\int_{0}^{T}\int_{\O_3} |\varrho_1|^2\dT.
\end{eqnarray} 	
In order to eliminate the last term in the previous inequality, we establish the energy estimates for the system \eqref{varrho} and we obtain:
\begin{equation}\label{l}
\begin{array}{llll}
	\dis \int_{0}^{T}\int_{\O_3}(|\varrho_1|^2+|\varrho_2|^2)\ dxdt
	\dis \leq C(\|c_1\|_{\infty}, \|c_2\|_{\infty}, \|d\|_{\infty},T) \left(\frac{\alpha_1^2}{\mu_1^2}+\frac{\alpha_2^2}{\mu_2^2}\right)\int_{0}^{T}\int_{\O_3} |\rho_{*}^{-2}\rho_1|^2\ dxdt.
\end{array}
\end{equation}
Using the definition of $\rho_{*}(t)$ given in Remark \ref{r}, we get $\dis \rho_{*}^{-2}\leq e^{s\varphi_{*}}\leq e^{s\varphi}$. Furthermore, $\Theta\in L^\infty((0,T))$ and  $\dis \frac{x^2}{a(x)}$ is non-decreasing on $(0,1]$. Then, the inequality \eqref{l} becomes
\begin{equation}\label{la}
	\begin{array}{llll}
		\dis \int_{0}^{T}\int_{\O_3}(|\varrho_1|^2+|\varrho_2|^2)\ dxdt
		\dis \leq C(\|c_1\|_{\infty}, \|c_2\|_{\infty}, \|d\|_{\infty},T) \left(\frac{\alpha_1^2}{\mu_1^2}+\frac{\alpha_2^2}{\mu_2^2}\right)\int_{0}^{T}\int_{\O_3}\Theta^3 \frac{x^2}{a(x)}e^{2s\varphi}|\rho_1|^2\ dxdt.
	\end{array}
\end{equation}
Combining \eqref{la} with \eqref{obser1a} and taking $\mu_i,\ i=1,2$ large enough, we obtain
	\begin{eqnarray}\label{obser1ab}
	\dis \mathcal{I}(\rho_1)+\mathcal{I}(\rho_2)+\mathcal{I}(\varrho_1)+\mathcal{I}(\varrho_2)
	\leq C_4\int_{0}^{T}\int_{\O_3} s^{15}\Theta^{33}e^{2s(4\Phi-3\varphi)}|\rho_1|^2\ \dT.
\end{eqnarray} 
Thanks to \eqref{ineqphi}, we have $4\Phi-3\varphi\leq 0$, then $\dis \Theta^{33}e^{2s(4\Phi-3\varphi)}\in L^\infty(Q)$. Furthermore, using the fact that, $\O_3\subset\omega$, we deduce the inequality \eqref{obser} and we complete the proof of Proposition \ref{prop5}.	
\end{proof}
\paragraph{}

To prove the needed observability inequality, we are going to improve the Carleman inequality \eqref{obser}. To this end, we modify the weight functions $\varphi$ and $\Theta$ defined in \eqref{functcarl} as follows:
\begin{equation}\label{phitil}
\widetilde{\varphi}(t,x)=\left\{
\begin{array}{rllll}
\dis\varphi\left(\frac{T}{2},x\right)\ \ \mbox{if}\ \ t\in \left[0,\frac{T}{2}\right],\\
\dis \varphi(t,x)\ \ \mbox{if}\ \ t\in \left[\frac{T}{2},T\right]
\end{array}
\right.
\end{equation}
and
\begin{equation}\label{Thetatil}
\widetilde{\Theta}(t)=\left\{
\begin{array}{rllll}
\dis\Theta\left(\frac{T}{2}\right)\ \ \mbox{if}\ \ t\in \left[0,\frac{T}{2}\right],\\
\dis \Theta(t)\ \ \mbox{if}\ \ t\in \left[\frac{T}{2},T\right].
\end{array}
\right.
\end{equation} Then in view of the definition of $\varphi$ and $\Theta$, the functions $\widetilde{\varphi}(.,x)$ and $\widetilde{\Theta}(\cdot)$ are non positive function of class $\mathcal{C}^1$ on $[0,T[$. From now on, we fix the parameter $s$. We have the following result.

\begin{prop} \label{pro}$ $
	
	  Under the assumptions of Proposition \ref{prop5}, there exist a positive constant\\ $C=C(C_4,\|b_1\|_{\infty}, \|b_2\|_{\infty},\|c_1\|_{\infty}, \|c_2\|_{\infty}, \|d\|_{\infty},\mu_1,\mu_2, T)>0$ and a positive weight function $\kappa$ such that every solution $\rho=(\rho_1,\rho_2)$ and $\psi^i=(\psi_1^i,\psi_2^i)$  of \eqref{rho} and \eqref{psi}, respectively, satisfy the following inequality:
	
	\begin{eqnarray}\label{obser3}
	\|\rho_1(0,\cdot)\|^2_{L^2(\Omega)}+\|\rho_2(0,\cdot)\|^2_{L^2(\Omega)}+\sum_{i=1}^{2}\int_{Q}\kappa^2(|\psi_1^i|^2+|\psi_2^i|^2)\,\dT
	\leq C\int_{0}^{T}\int_{\omega}|\rho_1|^2\,\dq,
	\end{eqnarray}
	where the constant $C_4$ is given by the Proposition \ref{prop5}.
\end{prop}

\begin{proof} 
	
	We proceed in two steps.\\
\textbf{Step 1.} We prove that there exist a constant $C=C(C_4,\|b_1\|_{\infty}, \|b_2\|_{\infty}, \|d\|_{\infty},T)>0$ such that 
\begin{equation}\label{observ}
	\begin{array}{llll}
		\dis  \|\rho_1(0,\cdot)\|^2_{L^2(\Omega)}+\|\rho_2(0,\cdot)\|^2_{L^2(\Omega)}+\sum_{i=1}^{2}\widetilde{\mathcal{I}}_{[0,T]}(\rho_i)+\widetilde{\mathcal{I}}_{[0,T]}(\varrho_i)
		\dis \leq\dis C\int_{0}^{T}\int_{\omega}|\rho_1|^2\,\dq,
	\end{array}
\end{equation}
	where $\widetilde{\mathcal{I}}(\cdot)$ is defined in the follow by \eqref{car3}.

	Let us introduce a function $\beta\in \mathcal{C}^1 ([0,T])$  such that
	\begin{equation}\label{hypobeta}
	0\leq \beta\leq 1,\ \beta(t)=1 \hbox{ for } t\in [0,T/2],\  \beta(t)=0 \hbox{ for } t\in [3T/4,T],\  |\beta^\prime(t)|\leq C/T.
	\end{equation}
	For any $(t,x)\in Q$, we set
	$$
	\begin{array}{lll}
	z_i(t,x)=\beta(t)e^{-r(T-t)}\rho_i(t,x),\ i=1,2,
	\end{array}
	$$
	where $r>0$. Then in view of \eqref{rhobis}, the function $z=(z_1,z_2)$  is solution of
	\begin{equation}\label{eq26z}
		\left\{
		\begin{array}{rllll}
			\dis -z_{1,t}-\left(a(x)z_{1,x}\right)_{x}+b_1z_1+dz_2  &=& \dis \beta e^{-r(T-t)}\varrho_1\chi_{\omega_d}-\beta^\prime e^{-r(T-t)} \rho_1 & \mbox{in}& Q,\\
			\dis -z_{2,t}-\left(a(x)z_{2,x}\right)_{x}+b_2z_2  &=&\dis \beta e^{-r(T-t)}\varrho_2\chi_{\omega_d}-\beta^\prime e^{-r(T-t)} \rho_2& \mbox{in}& Q,\\
			\dis z_1(t,0)=z_1(t,1)=z_2(t,0)=z_2(t,1)&=&0& \mbox{on}& (0,T), \\
			\dis z_1(T,\cdot)=z_2(T,\cdot)&=&0 &\mbox{in}&\Omega.
		\end{array}
		\right.
	\end{equation}
	Using the classical energy estimates for the system \eqref{eq26z} and using the definition of $\beta$ and $z$, we get
	$$
	\begin{array}{llll}
	\dis \|\rho_1(0,\cdot)\|^2_{L^2(\Omega)}+\|\rho_2(0,\cdot)\|^2_{L^2(\Omega)}+\int_{0}^{T/2}\int_{\Omega}(|\rho_1|^2+|\rho_2|^2)\ dxdt+\int_{0}^{T/2}\int_{\Omega}a(x)(|\rho_{1,x}|^2+|\rho_{2,x}|^2)\ dxdt\\
	\dis \leq C(\|b_1\|_{\infty}, \|b_2\|_{\infty}, \|d\|_{\infty},T)
	\left(\int_0^{3T/4}\int_\Omega (|\varrho_1|^2+|\varrho_2|^2)\ \dq
	+\int_{T/2}^{3T/4}\int_\Omega( |\rho_1|^2+ |\rho_2|^2)\ \dq\right).
	\end{array}
	$$
	The functions $\widetilde{\varphi}$ and $\widetilde{\Theta}$ defined by \eqref{phitil} and \eqref{Thetatil}, respectively, have lower and upper bounds for $(t,x)\in  [0,T/2]\times \Omega$. Furthermore,
	due to the hypothesis \eqref{k}, we have that $a\in \mathcal{C}([0,1])$ and $a>0\ \mbox{in}\ (0,1]$. Therefore, there exist positive constants $\alpha_1$ and $\alpha_2$ such that $ a(x)\geq \alpha_1$ and $ \dis \frac{x^2}{a(x)}\geq \alpha_2$.
	Then, we introduce the corresponding weight functions in the above expression and we get
	\begin{equation}\label{pou}
	\begin{array}{llll}
	\dis \|\rho_1(0,\cdot)\|^2_{L^2(\Omega)}+\|\rho_2(0,\cdot)\|^2_{L^2(\Omega)}+\widetilde{\mathcal{I}}_{[0,T/2]}(\rho_1)+\widetilde{\mathcal{I}}_{[0,T/2]}(\rho_2)\\ \dis \leq C(\|b_1\|_{\infty}, \|b_2\|_{\infty}, \|d\|_{\infty},T)
	\left(\int_0^{3T/4}\int_\Omega (|\varrho_1|^2+|\varrho_2|^2)\ \dq
	+\int_{T/2}^{3T/4}\int_\Omega( |\rho_1|^2+ |\rho_2|^2)\ \dq\right),
	\end{array}
	\end{equation}
	where 
	\begin{equation}\label{car3}
	\widetilde{\mathcal{I}}_{[a,b]}(l)=\int_a^b\int_{\Omega}\widetilde{\Theta}^3\frac{x^2}{a(x)}e^{2s\widetilde{\varphi}}|l|^2\ \dq+\int_a^b\int_{\Omega}\widetilde{\Theta}a(x)e^{2s\widetilde{\varphi}}|l_x|^2\ \dq.
	\end{equation}
	Adding the term $\widetilde{\mathcal{I}}_{[0,T/2]}(\varrho_1)+\widetilde{\mathcal{I}}_{[0,T/2]}(\varrho_2)$ on both sides of inequality \eqref{pou}, we have
		\begin{equation}\label{pou1}
		\begin{array}{llll}
			\dis \|\rho_1(0,\cdot)\|^2_{L^2(\Omega)}+\|\rho_2(0,\cdot)\|^2_{L^2(\Omega)}+\widetilde{\mathcal{I}}_{[0,T/2]}(\rho_1)+\widetilde{\mathcal{I}}_{[0,T/2]}(\rho_2)+\widetilde{\mathcal{I}}_{[0,T/2]}(\varrho_1)+\widetilde{\mathcal{I}}_{[0,T/2]}(\varrho_2)\\ 
			\dis \leq C(\|b_1\|_{\infty}, \|b_2\|_{\infty}, \|d\|_{\infty},T)
			\left(\int_0^{3T/4}\int_\Omega (|\varrho_1|^2+|\varrho_2|^2)\ \dq
			+\int_{T/2}^{3T/4}\int_\Omega( |\rho_1|^2+ |\rho_2|^2)\ \dq\right)\\
			\dis +\widetilde{\mathcal{I}}_{[0,T/2]}(\varrho_1)+\widetilde{\mathcal{I}}_{[0,T/2]}(\varrho_2).
		\end{array}
	\end{equation}
	In order to eliminate the term $\widetilde{\mathcal{I}}_{[0,T/2]}(\varrho_1)+\widetilde{\mathcal{I}}_{[0,T/2]}(\varrho_2)$ in the right hand side of \eqref{pou1}, we use the classical energy estimates for the system \eqref{varrho} and we obtain:
	$$
	\begin{array}{llll}
		\dis \int_{0}^{T/2}\int_{\Omega}(|\varrho_1|^2+|\varrho_2|^2)\ dxdt+\int_{0}^{T/2}\int_{\Omega}a(x)(|\varrho_{1,x}|^2+|\varrho_{2,x}|^2)\ dxdt\\
		\dis \leq C(\|c_1\|_{\infty}, \|c_2\|_{\infty}, \|d\|_{\infty},T) \left(\frac{\alpha_1^2}{\mu_1^2}+\frac{\alpha_2^2}{\mu_2^2}\right)\int_{0}^{T/2}\int_\Omega |\rho_{*}^{-2}\rho_1|^2\ dxdt\\
		\dis \leq C(\|c_1\|_{\infty}, \|c_2\|_{\infty}, \|d\|_{\infty},T) \left(\frac{\alpha_1^2}{\mu_1^2}+\frac{\alpha_2^2}{\mu_2^2}\right)\int_{0}^{T/2}\int_\Omega e^{2s\varphi}|\rho_1|^2\ dxdt,
	\end{array}
	$$
	where $C$ is independent of $\mu_i,\ i=1,2$. The functions $\widetilde{\varphi}$ and $\widetilde{\Theta}$ have lower and upper bounds for $(t,x)\in  [0,T/2]\times \Omega$. Moreover, the function $\dis \frac{x^2}{a(x)}$ is non-decreasing on $(0;1]$ and $\dis \frac{x^2}{a(x)}\geq \alpha_2>0$ in $(0,1]$. Then, from the previous inequality and using the fact that $\dis e^{2s\varphi}\leq 1$ for all $(t,x)\in Q$, we obtain 
	\begin{equation}\label{pou2}
	\begin{array}{llll}
	\dis \widetilde{\mathcal{I}}_{[0,T/2]}(\varrho_1)+\widetilde{\mathcal{I}}_{[0,T/2]}(\varrho_2)\\
	\dis \leq C(\|c_1\|_{\infty)}, \|c_2\|_{\infty}, \|d\|_{\infty},T) \left(\frac{\alpha_1^2}{\mu_1^2}+\frac{\alpha_2^2}{\mu_2^2}\right)\int_{0}^{T/2}\int_\Omega \widetilde{\Theta}^3\frac{x^2}{a(x)}e^{2s\widetilde{\varphi}}|\rho_1|^2\ dxdt.
     \end{array}
	\end{equation}
	Replacing \eqref{pou2} in \eqref{pou1} and taking $\mu_i,\ i=1,2$ large enough, we obtain
	\begin{equation}\label{pou3}
		\begin{array}{llll}
			\dis \|\rho_1(0,\cdot)\|^2_{L^2(\Omega)}+\|\rho_2(0,\cdot)\|^2_{L^2(\Omega)}+\widetilde{\mathcal{I}}_{[0,T/2]}(\rho_1)+\widetilde{\mathcal{I}}_{[0,T/2]}(\rho_2)+\widetilde{\mathcal{I}}_{[0,T/2]}(\varrho_1)+\widetilde{\mathcal{I}}_{[0,T/2]}(\varrho_2)\\ 
			\dis \leq C(\|b_1\|_{\infty}, \|b_2\|_{\infty}, \|d\|_{\infty},T)
		\int_{T/2}^{3T/4}\int_\Omega ( |\rho_1|^2+ |\rho_2|^2+|\varrho_1|^2+|\varrho_2|^2)\ \dq.
		\end{array}
	\end{equation}
	The functions $\varphi$ and $\Theta$ defined in \eqref{functcarl} have the lower and upper bounds for $(t,x)\in  [T/2,3T/4]\times \Omega$. Moreover, the function $\dis \frac{x^2}{a(x)}$ is non-decreasing on $(0,1]$. Using \eqref{obser}, the relation \eqref{pou3} becomes
	\begin{equation}\label{pou4}
		\begin{array}{llll}
			&&\dis \|\rho_1(0,\cdot)\|^2_{L^2(\Omega)}+\|\rho_2(0,\cdot)\|^2_{L^2(\Omega)}+\widetilde{\mathcal{I}}_{[0,T/2]}(\rho_1)+\widetilde{\mathcal{I}}_{[0,T/2]}(\rho_2)+\widetilde{\mathcal{I}}_{[0,T/2]}(\varrho_1)+\widetilde{\mathcal{I}}_{[0,T/2]}(\varrho_2)\\ 
			&&\dis \leq\dis  C(\|b_1\|_{\infty}, \|b_2\|_{\infty}, \|d\|_{\infty},T)
			\sum_{i=1}^{2}\left(\mathcal{I}_{[T/2;3T/4]}(\rho_i)+\mathcal{I}_{[T/2;3T/4]}(\varrho_i)\right)\\
			&&\leq\dis C(C_4,\|b_1\|_{\infty}, \|b_2\|_{\infty}, \|d\|_{\infty},T)\int_{0}^{T}\int_{\omega}|\rho_1|^2\,\dq,
		\end{array}
	\end{equation}
	where $\mathcal{I}(\cdot)$ is defined by \eqref{I} and the constant $C_4$ is defined in the Proposition \ref{prop5}.
	
	On the other hand, since $\Theta=\widetilde{\Theta}$ and $\varphi=\tilde{\varphi}$ in $[T/2,T]\times \Omega$, we use again estimate \eqref{obser} and we obtain
	\begin{equation}\label{pou5}
	\begin{array}{llll}
	\dis \sum_{i=1}^{2}\widetilde{\mathcal{I}}_{[T/2,T]}(\rho_i)+\widetilde{\mathcal{I}}_{[T/2,T]}(\varrho_i)&\leq&\dis\sum_{i=1}^{2} \mathcal{I}_{[T/2,T]}(\rho_i)+\mathcal{I}_{[T/2,T]}(\varrho_i)\\
	&\leq&\dis C(C_4,\|b_1\|_{\infty}, \|b_2\|_{\infty}, \|d\|_{\infty},T)\int_{0}^{T}\int_{\omega}|\rho_1|^2\,\dq.
	\end{array}
	\end{equation}
	Adding \eqref{pou4} and \eqref{pou5}, we get
	\begin{equation*}\label{}
		\begin{array}{llll}
			\dis  \|\rho_1(0,\cdot)\|^2_{L^2(\Omega)}+\|\rho_2(0,\cdot)\|^2_{L^2(\Omega)}+\sum_{i=1}^{2}\widetilde{\mathcal{I}}_{[0,T]}(\rho_i)+\widetilde{\mathcal{I}}_{[0,T]}(\varrho_i)
			\dis \leq\dis C\int_{0}^{T}\int_{\omega}|\rho_1|^2\,\dq,
		\end{array}
	\end{equation*}
	 and then, we deduce the estimation \eqref{observ}.\\
	\textbf{Step 2.}
	Now, we prove that there exist a constant $C=C(C_4,\|b_1\|_{\infty}, \|b_2\|_{\infty},\|c_1\|_{\infty}, \|c_2\|_{\infty}, \|d\|_{\infty},\mu_1,\mu_2, T)>0$ and a positive weight function $\kappa$ such that
		\begin{eqnarray}\label{observ1}
		\sum_{i=1}^{2}\int_{Q}\kappa^2(|\psi_1^i|^2+|\psi_2^i|^2)\,\dT
		\leq C\int_{0}^{T}\int_{\omega}|\rho_1|^2\,\dq.
	\end{eqnarray}
	
	Let us introduce the function
	\begin{eqnarray}\label{berlin}
	\dis \hat{\varphi}(t)=\min_{x\in \Omega}\widetilde{\varphi}(t,x)
	\end{eqnarray}
and define the weight function $\kappa$  by:

\begin{eqnarray}\label{deftkappa}
	\dis \kappa(t)=e^{s\hat{\varphi}(t)}\in L^\infty(0,T).
\end{eqnarray}
	Then $\kappa$ is a positive function of class $\mathcal{C}^1$ on $[0,T)$. Furthermore, $\dis \frac{\partial \hat{\varphi}}{\partial t}$ is also a positive function on $(0,T)$. 
	Now, multiplying the first equation and the second equation of \eqref{psi} by $\kappa^2\psi_1^i$ and $\kappa^2\psi_2^i$, respectively and integrating by parts over $\Omega$, we obtain that 
	\begin{equation*}
	\begin{array}{rlll}
	\dis \frac{1}{2}\frac{d}{dt}\int_{\Omega} \kappa^2|\psi_1^i|^2\ dx+\int_{\Omega} \kappa^2a(x)|\psi_{1,x}^i|^2\ dx
	\dis &=&\dis -\int_{\Omega}\kappa^2\, c_1\,|\psi_1^i|^2\ dx
	\dis -\frac{1}{\mu_i}\int_{\omega_i} \kappa^2\rho_{*}^{-2}\rho_1\psi_1^i\ dx+s\dis \int_{\Omega} \kappa^2\frac{\partial\hat{\varphi}}{\partial t}|\psi_1^i|^2\ dx\\
	&\leq&\dis \left(\|c_1\|_{\infty}+\frac{1}{2}\right)\int_{\Omega} \kappa^2|\psi_1^i|^2\ dx+\frac{1}{2\mu_i^2}\int_{\omega_i} \kappa^2|\rho_1|^2\ dx
	\end{array}
	\end{equation*}
and
\begin{equation*}
 \begin{array}{rlll}
	\dis \frac{1}{2}\frac{d}{dt}\int_{\Omega} \kappa^2|\psi_2^i|^2\ dx+\int_{\Omega} \kappa^2a(x)|\psi_{2,x}^i|^2\ dx
	\dis &=&\dis-\int_{\Omega}\kappa^2\, c_2\,|\psi_2^i|^2\ dx
	\dis -\int_{\Omega}\kappa^2\, d\,|\psi_1^i \psi_2^i|^2\ dx+s\dis \int_{\Omega} \kappa^2\frac{\partial\hat{\varphi}}{\partial t}|\psi_2^i|^2\ dx\\
	&\leq&\dis \left(\|c_2\|_{\infty}+\frac{\|d\|^2_{\infty}}{2}\right)\int_{\Omega} \kappa^2|\psi_1^i|^2\ dx+\frac{1}{2\mu_i^2}\int_{\omega_i} \kappa^2|\psi_2^i|^2\ dx.
\end{array}
\end{equation*}
We obtain the two previous inequalities using the fact that $\dis \frac{\partial \hat{\varphi}}{\partial t}$ is a positive function on $[0,T)$. 
Adding the two previous inequalities, we obtain
\begin{equation*}
	\begin{array}{rlll}
		\dis \frac{d}{dt}\left(\int_{\Omega} \kappa^2(|\psi_1^i|^2+|\psi_2^i|^2)\ dx\right)
		\leq\dis C\int_{\Omega} \kappa^2(|\psi_1^i|^2+|\psi_2^i|^2)\ dx+\frac{1}{\mu_i^2}\int_{\omega_i} \kappa^2|\rho_1|^2\ dx,
	\end{array}
\end{equation*}
where $C=\left(\|c_1\|_{\infty}, \|c_2\|_{\infty}, \|d\|_{\infty}\right)$.
	Using Gronwall's Lemma and the fact that $\psi_1^i(x,0)=\psi_2^i(x,0)=0$ for $x\in \Omega$, we obtain that
	\begin{equation}\label{berlin2}
	\begin{array}{rlll}
	\dis \int_{\Omega} \kappa^2(|\psi_1^i|^2+|\psi_2^i|^2)\ dx\leq C\int_Q  \kappa^2|\rho_1|^2\ dx,\ \forall t\in [0,T],
	\end{array}
	\end{equation}
where $C=\left(\|c_1\|_{\infty}, \|c_2\|_{\infty}, \|d\|_{\infty}, T,\mu_1,\mu_2\right)$.
Using the definition of $\hat{\varphi}$ and $\kappa$ given by \eqref{berlin} and \eqref{deftkappa}, respectively, we have
\begin{equation}\label{ten}
	\kappa^2(t)\leq e^{2s\widetilde{\varphi}(t,x)},\ \ \forall x\in \Omega.
\end{equation}
Thanks to the fact that $\dis \widetilde{\Theta}^{-1}\in L^\infty(0,T)$ and that the function $\dis \frac{a(x)}{x^2}$ is non-decreasing on $(0,1]$, using \eqref{ten}, we have
	\begin{equation*}
	\begin{array}{rlll}
	\dis \int_{Q} \kappa^2(|\psi_1^i|^2+|\psi_2^i|^2)\ dx\leq \int_{Q} \widetilde{\Theta}^3\frac{x^2}{a(x)}e^{2s\widetilde{\varphi}}|\rho_1|^2\ \dT,
	\end{array}
	\end{equation*}
	which combining with \eqref{berlin2} and  \eqref{observ} yields
	\begin{equation*}
	\begin{array}{rlll}
	\dis \int_{Q} \kappa^2(|\psi_1^i|^2+|\psi_2^i|^2)\ dx\leq C\int_0^T\int_{\omega} |\rho_1|^2 \dT,
	\end{array}
	\end{equation*}
	where $C=\left(C_4,\|b_1\|_{\infty}, \|b_2\|_{\infty},\|c_1\|_{\infty}, \|c_2\|_{\infty}, \|d\|_{\infty}, T,\mu_1,\mu_2\right)$.
	Adding this latter inequality with \eqref{observ}, we deduce \eqref{obser3}.	
\end{proof}

\section{Null controllability problem}\label{null}	

In this section, we will end the proof of Theorem \ref{theolinear}. The proof is inspired by well-known results on the controllability of non linear systems where controllability of linear systems and suitable fixed point arguments are the main ingredients.

\subsection{Null controllability of an auxiliary linear system }
Here, we prove the null controllability of a linearized version of \eqref{yop}-\eqref{pop}. In fact, for given  $b_1,b_2,c_1,c_2,d\in L^\infty(Q)$, we consider the linear systems
\begin{equation}\label{ylin}
	\left\{
	\begin{array}{rllll}
		\dis y_{1,t}-\left(a(x)y_{1,x}\right)_{x}+b_1y_1  &=&\dis h\chi_{\omega}-\frac{1}{\mu_1}\rho_{*}^{-2}p_1^1\chi_{\omega_1}-\frac{1}{\mu_2}\rho_{*}^{-2}p_1^2\chi_{\omega_2}& \mbox{in}& Q,\\
		\dis y_{2,t}-\left(a(x)y_{2,x}\right)_{x}+b_2y_2+dy_1  &=&0& \mbox{in}& Q,\\
		\dis y_1(t,0)=y_1(t,1)=y_2(t,0)=y_2(t,1)&=&0& \mbox{on}& (0,T), \\
		\dis y_1(0,\cdot)=y_1^0,\ \ y_2(0,\cdot)=y_2^0&& &\mbox{in}&\Omega
	\end{array}
	\right.
\end{equation}
and 
\begin{equation}\label{plin}
	\left\{
	\begin{array}{rllll}
		\dis -p_{1,t}^i-\left(a(x)p^i_{1,x}\right)_{x}+c_1p_1^i+dp_2^i   &=&\alpha_i\left(y_1-y_{1,d}^i\right)\chi_{\omega_{d}}& \mbox{in}& Q,\\
		\dis -p_{2,t}^i-\left(a(x)p^i_{2,x}\right)_{x}+c_2p_2^i  &=&\alpha_i\left(y_2-y_{2,d}^i\right)\chi_{\omega_{d}}& \mbox{in}& Q,\\
		\dis p_1^i(t,0)=p_1^i(t,1)=p_2^i(t,0)=p_2^i(t,1)&=&0& \mbox{on}& (0,T), \\
		\dis p_1^i(T,\cdot)= p_2^i(T,\cdot)&=&0 &\mbox{in}&\Omega,
	\end{array}
	\right.
\end{equation}
and the corresponding adjoint systems \eqref{psi}-\eqref{rhobis}.
Thanks to the Proposition \ref{pro}, we will able to prove the null controllability of \eqref{ylin}-\eqref{plin}. We have the following result.

\begin{prop}\label{linear}$ $
	
	Suppose that \eqref{od} holds, the $\mu_i,\ i=1,2$ are large enough, $a(\cdot)$ satisfies \eqref{k} and the coefficients $b_1$, $b_2$, $c_1$, $c_2$ and $d$ belong to $L^\infty(Q)$. If the condition \eqref{d} is satisfied, then for any $y^0\in [L^2(\Omega)]^2$ and $\kappa^{-2}y^i_{j,d}\in L^2(Q)$, there exists a control $\bar{h}\in L^2(\omega_T)$ such that the corresponding solutions to \eqref{ylin} and \eqref{plin} satisfies \eqref{mainobj}. Furthermore, there exists a positive constant $C$ depending on 
	$C_4,\|b_1\|_{\infty}, \|b_2\|_{\infty},\|c_1\|_{\infty},$ $\|c_2\|_{\infty}, \|d\|_{\infty},\mu_1,\mu_2$ and $T$ such that
	\begin{equation}\label{mon10theo}
		\begin{array}{ccc}
			\dis \|\bar{h}\|_{L^2(\omega_T)}&\leq&C
			\dis  \left(\alpha_1^2\sum_{i=1}^2\left\|\kappa^{-1} y^1_{i,d}\right\|^2_{L^2(\omega_{d,T})}+\alpha_2^2\sum_{i=1}^2\left\|\kappa^{-1} y^2_{i,d}\right\|^2_{L^2(\omega_{d,T})}+ \|y_1^0\|^2_{L^2(\Omega)}+ \|y_2^0\|^2_{L^2(\Omega)}\right)^{1/2}.
		\end{array}
	\end{equation}
\end{prop}

\begin{proof}

To prove this null controllability problem, we proceed in three steps using a penalization method.\\
\noindent\textbf{Step 1.} For any $\varepsilon >0$, we consider the following cost function:
\begin{equation}\label{defJ}
	J_{\varepsilon}(h)=\dis \frac{1}{2\varepsilon}\int_{\Omega}\left(|y_1(T,\cdot)|^2+|y_2(T,\cdot)|^2\right)\ dx+
	\frac{1}{2}\int_{\omega_T}|h|^2\ \dq.
\end{equation}
Then we consider the optimal control problem: 
\begin{equation}\label{opt}
	\inf_{\atop h\in L^2(\omega_T)}J_{\varepsilon}(h).
\end{equation}
We can prove that $J_{\varepsilon}$ is continuous, coercive and strictly convex. Then, the optimization problem \eqref{opt}  admits a unique solution $h_\varepsilon$  and arguing as in \cite{djomegne2021}, we prove that 
\begin{equation}\label{hgeps}
	h_{\varepsilon}=\rho_{1\varepsilon}\ \ \mbox{in}\ \ \omega_T,
\end{equation}	
with $(\rho_\varepsilon,\psi^i_\varepsilon)$ is the solution of the following systems
\begin{equation}\label{rhoeps}
	\left\{
	\begin{array}{rllll}
		\dis -\rho_{1\varepsilon,t}-\left(a(x)\rho_{1\varepsilon,x}\right)_{x}+b_1\rho_{1\varepsilon}+d\rho_{2\varepsilon}  &=&\dis (\alpha_1\psi_{1\varepsilon}^1+\alpha_2\psi_{1\varepsilon}^2)\chi_{\omega_{d}}& \mbox{in}& Q,\\
		\dis -\rho_{2\varepsilon,t}-\left(a(x)\rho_{2\varepsilon,x}\right)_{x}+b_2\rho_{2\varepsilon}  &=&\dis (\alpha_1\psi_{2\varepsilon}^1+\alpha_2\psi_{2\varepsilon}^2)\chi_{\omega_{d}}& \mbox{in}& Q,\\
		\dis \rho_{1\varepsilon}(t,0)=\rho_{1\varepsilon}(t,1)=\rho_{2\varepsilon}(t,0)=\rho_{2\varepsilon}(t,1)&=&0& \mbox{on}& (0,T), \\
		\dis \rho_{1\varepsilon}(T,\cdot)=-\dis \frac{1}{\varepsilon}y_{1\varepsilon} (T,\cdot)\ \ \rho_{2\varepsilon}(T,\cdot)=-\dis \frac{1}{\varepsilon}y_{2\varepsilon} (T,\cdot)&& &\mbox{in}&\Omega
	\end{array}
	\right.
\end{equation}
and
\begin{equation}\label{psieps}
	\left\{
	\begin{array}{rllll}
		\dis \psi_{1\varepsilon,t}^i-\left(a(x)\psi^i_{1\varepsilon,x}\right)_{x}+c_1\psi_{1\varepsilon}^i   &=&\dis-\frac{1}{\mu_i}\rho_{*}^{-2}\rho_{1\varepsilon}\chi_{\omega_{i}}& \mbox{in}& Q,\\
		\dis \psi_{2\varepsilon,t}^i-\left(a(x)\psi^i_{2\varepsilon,x}\right)_{x}+c_2\psi_{2\varepsilon}^i+d\psi_{1\varepsilon}^i   &=&0& \mbox{in}& Q,\\
		\dis \psi_{1\varepsilon}^i(t,0)=\psi_{1\varepsilon}^i(t,1)=\psi_{2\varepsilon}^i(t,0)=\psi_{2\varepsilon}^i(t,1)&=&0& \mbox{on}& (0,T), \\
		\dis \psi_{1\varepsilon}^i(0,\cdot)= \psi_{2\varepsilon}^i(0,\cdot)&=&0 &\mbox{in}&\Omega,
	\end{array}
	\right.
\end{equation}
where $(y_\varepsilon, p^i_\varepsilon)$ is the solution of systems \eqref{ylin}-\eqref{plin} associated to the control $h_\varepsilon$.\\
\textbf{Step 2.} 
If we  multiply  the first and the second equation of \eqref{rhoeps} by $y_{1\varepsilon}$ and $y_{2\varepsilon}$ respectively, then we  multiply  the first and the second equation of \eqref{psieps} by $p^i_{1\varepsilon}$ and $p^i_{2\varepsilon}$ respectively, we integrate by parts over $Q$ and add the different equations, we obtain from \eqref{hgeps}
$$
\begin{array}{rlll}
\dis \|h_\varepsilon\|^2_{L^2(\omega_T)}+\frac{1}{\varepsilon}\|y_{1\varepsilon}(T,\cdot)\|^2_{L^2(\Omega)}+\frac{1}{\varepsilon}\|y_{2\varepsilon}(T,\cdot)\|^2_{L^2(\Omega)}
=\dis-\int_{\Omega}y_1^0\rho_{1\varepsilon}(0,\cdot)\ dx-\int_{\Omega}y_2^0\rho_{2\varepsilon}(0,\cdot)\ dx+\\
\dis \alpha_1\sum_{i=1}^2\int_{\omega_{d,T}}y^1_{i,d}\psi^1_{i\varepsilon} \dT+\alpha_2\sum_{i=1}^2\int_{\omega_{d,T}}y^2_{i,d}\psi^2_{i\varepsilon} \dT.	
\end{array}
$$
Using the Young inequality, we obtain 
\begin{equation}\label{rosnymon10}
	\begin{array}{ccc}
		&&\dis \|h_\varepsilon\|^2_{L^2(\omega_T)}+\frac{1}{\varepsilon}\|y_{1\varepsilon}(T,\cdot)\|^2_{L^2(\Omega)}+\frac{1}{\varepsilon}\|y_{2\varepsilon}(T,\cdot)\|^2_{L^2(\Omega)}\\
		&&\leq
		\dis  \left(\alpha_1^2\sum_{i=1}^2\left\|\kappa^{-1} y^1_{i,d}\right\|^2_{L^2(\omega_{d,T})}+\alpha_2^2\sum_{i=1}^2\left\|\kappa^{-1} y^2_{i,d}\right\|^2_{L^2(\omega_{d,T})}+ \|y_1^0\|^2_{L^2(\Omega)}+ \|y_2^0\|^2_{L^2(\Omega)}\right)^{1/2}\times\\ 
		&&\dis \left( \sum_{i=1}^2 \left\|\kappa\psi_{i\varepsilon}^1\right\|^2_{L^2(Q)}+\sum_{i=1}^2 \left\|\kappa\psi_{i\varepsilon}^2\right\|^2_{L^2(Q)}+
		\|\rho_{1\varepsilon}(0,\cdot)\|^2_{L^2(\Omega)}+\|\rho_{2\varepsilon}(0,\cdot)\|^2_{L^2(\Omega)}\right)^{1/2}.
	\end{array}
\end{equation}
Using the observability inequality \eqref{obser3}, we deduce from \eqref{rosnymon10} the existence of a positive constant $C$ depending on $C_4,\|b_1\|_{\infty}, \|b_2\|_{\infty},\|c_1\|_{\infty}, \|c_2\|_{\infty},$ $\|d\|_{\infty},\mu_1,\mu_2$ and $T$ such that
\begin{equation}\label{mon10}
	\begin{array}{ccc}
		\dis \|h_\varepsilon\|_{L^2(\omega_T)}\leq C
		\dis  \left(\alpha_1^2\sum_{i=1}^2\left\|\kappa^{-1} y^1_{i,d}\right\|^2_{L^2(\omega_{d,T})}+\alpha_2^2\sum_{i=1}^2\left\|\kappa^{-1} y^2_{i,d}\right\|^2_{L^2(\omega_{d,T})}+ \|y_1^0\|^2_{L^2(\Omega)}+ \|y_2^0\|^2_{L^2(\Omega)}\right)^{1/2},
	\end{array}
\end{equation}
\begin{equation}\label{mon10a}
	\begin{array}{ccc}
		\dis \|y_{1\varepsilon}(T,\cdot)\|_{L^2(\Omega)}\leq C
		\dis  \sqrt{\varepsilon}\left(\alpha_1^2\sum_{i=1}^2\left\|\kappa^{-1} y^1_{i,d}\right\|^2_{L^2(\omega_{d,T})}+\alpha_2^2\sum_{i=1}^2\left\|\kappa^{-1} y^2_{i,d}\right\|^2_{L^2(\omega_{d,T})}+ \|y_1^0\|^2_{L^2(\Omega)}+ \|y_2^0\|^2_{L^2(\Omega)}\right)^{1/2}
	\end{array}
\end{equation}
and
\begin{equation}\label{mon10b}
	\begin{array}{ccc}
		\dis \|y_{2\varepsilon}(T,\cdot)\|_{L^2(\Omega)}\leq C
		\dis  \sqrt{\varepsilon}\left(\alpha_1^2\sum_{i=1}^2\left\|\kappa^{-1} y^1_{i,d}\right\|^2_{L^2(\omega_{d,T})}+\alpha_2^2\sum_{i=1}^2\left\|\kappa^{-1} y^2_{i,d}\right\|^2_{L^2(\omega_{d,T})}+ \|y_1^0\|^2_{L^2(\Omega)}+ \|y_2^0\|^2_{L^2(\Omega)}\right)^{1/2}.
	\end{array}
\end{equation}
Using \eqref{mon10}-\eqref{mon10b} and systems \eqref{ylin}-\eqref{plin} associated to the control $h_\varepsilon$ given by \eqref{hgeps}, we can extract subsequences still denoted by $h_\varepsilon,\ y_\varepsilon $ and $p^i_\varepsilon$ such that when $\varepsilon \rightarrow 0$, we have
\begin{subequations}\label{convergence2}
	\begin{alignat}{9}
		h_\varepsilon&\rightharpoonup& \bar{h}&\text{ weakly in }&L^{2}(\omega_T), \label{18k}\\
		y_{i\varepsilon}&\rightharpoonup& y_i&\text{  weakly in }&L^2((0,T);H^1_a(\Omega)),\ i=1,2, \label{}\\
		p^i_{j\varepsilon}&\rightharpoonup& p^i_j&\text{  weakly in }&L^2((0,T);H^1_a(\Omega)),\ i,j=1,2, \label{19}\\
		y_{i\varepsilon}(T,\cdot)&\longrightarrow&0&\hbox{ strongly in }&\ L^2(\Omega),\ i=1,2.\label{all5}
	\end{alignat}
\end{subequations}

Arguing as in \cite{djomegne2018,djomegne2021}, using convergences \eqref{convergence2}, we prove that $(y,\ p^i)$ is a solution of \eqref{ylin}-\eqref{plin} corresponding to the control $\bar{h}$ and also $y$ satisfies \eqref{mainobj}. Furthermore, using the convergence \eqref{18k}, we have that $\bar{h}$ satisfies \eqref{mon10theo}.

\end{proof}

\subsection{Proof of Theorem \ref{theolinear}}
We have proved in Proposition \ref{quasi} and Theorem \ref{nash} that the Nash equilibrium for $(J_1,J_2)$ given by \eqref{all16}, $(\hat{v}_1, \hat{v}_2)$ is characterised by \eqref{vop}-\eqref{pop}. In Proposition \ref{linear}, we proved that the linear systems \eqref{ylin}-\eqref{plin} is null controllable at time $t=T$. We are now going to prove that, there exists a control $\bar{h}\in L^2(\omega_T)$ such that the solution of \eqref{yop}-\eqref{pop} satisfies \eqref{mainobj}.

We define $W=[L^2((0,T);H^1_{a}(\Omega))]^2$. For every $w\in W$, we consider the linearized system for \eqref{yop}-\eqref{pop} 
\begin{equation}\label{ylina}
	\left\{
	\begin{array}{rllll}
		\dis y_{1,t}-\left(a(x)y_{1,x}\right)_{x}+b_1^{w}y_1  &=&\dis h\chi_{\omega}-\frac{1}{\mu_1}\rho_{*}^{-2}p_1^1\chi_{\omega_1}-\frac{1}{\mu_2}\rho_{*}^{-2}p_1^2\chi_{\omega_2}& \mbox{in}& Q,\\
		\dis y_{2,t}-\left(a(x)y_{2,x}\right)_{x}+b_2^{w}y_2+dy_1  &=&0& \mbox{in}& Q,\\
		\dis y_1(t,0)=y_1(t,1)=y_2(t,0)=y_2(t,1)&=&0& \mbox{on}& (0,T), \\
		\dis y_1(0,\cdot)=y_1^0,\ \ y_2(0,\cdot)=y_2^0&& &\mbox{in}&\Omega
	\end{array}
	\right.
\end{equation}
and 
\begin{equation}\label{plina}
	\left\{
	\begin{array}{rllll}
		\dis -p_{1,t}^i-\left(a(x)p^i_{1,x}\right)_{x}+c_1^{w}p_1^i+dp_2^i   &=&\alpha_i\left(y_1-y_{1,d}^i\right)\chi_{\omega_{i,d}}& \mbox{in}& Q,\\
		\dis -p_{2,t}^i-\left(a(x)p^i_{2,x}\right)_{x}+c_2^{w}p_2^i  &=&\alpha_i\left(y_2-y_{2,d}^i\right)\chi_{\omega_{i,d}}& \mbox{in}& Q,\\
		\dis p_1^i(t,0)=p_1^i(t,1)=p_2^i(t,0)=p_2^i(t,1)&=&0& \mbox{on}& (0,T), \\
		\dis p_1^i(T,\cdot)= p_2^i(T,\cdot)&=&0 &\mbox{in}&\Omega,
	\end{array}
	\right.
\end{equation}
where 
\begin{equation}
	b_i^{w}=\int_{0}^{1}F_i^\prime(\sigma w_i)d\sigma,\ \ c_i^{w}=F_i^\prime(w_i),\ \ i=1,2.	
\end{equation}
Observe that systems \eqref{ylina}-\eqref{plina} are of the form \eqref{ylin}-\eqref{plin} with $\dis b_i=b_i^{w}=\int_{0}^{1}F_i^\prime(\sigma w_i)d\sigma,\ \ i=1,2$ and $c_i=c_i^{w}=F_i^\prime(w_i),\ \ i=1,2$. 

Thanks to the hypothesis of $F_1$ and $F_2$ given by \eqref{lip}, we have that $b_1^{w},b_2^{w},c_1^{w}$ and $c_2^{w}$ belong to $L^\infty(Q)$. In view of Proposition \ref{linear}, there exits a control $\bar{h}(w)\in L^2(\omega_T)$ such that the solution $(y, p^i)$ to \eqref{ylina}-\eqref{plina} with $\bar{h}=\bar{h}(w)$ satisfies \eqref{mainobj}.

Combining \eqref{esty1y2}, \eqref{mon10theo} and \eqref{v}, we obtain
\begin{equation}\label{y}
		\dis \|y_1\|_{L^2((0,T); H^1_a(\Omega))}+\|y_2\|_{L^2((0,T); H^1_a(\Omega))}
		\leq 
		C\left(\|y_1^0\|_{L^2(\Omega)}+\|y_2^0\|_{L^2(\Omega)}\right).
\end{equation} 	

For every $w\in W$, we define

\begin{equation}
I(w)=\left\{\bar{h}\in L^2(\omega_T),\ (y,p^i)\ \mbox{solution of}\ \eqref{ylina}-\eqref{plina}\ \mbox{satisfies}\ \eqref{mainobj}\ \mbox{with}\ \bar{h}\ \mbox{verifying}\ \eqref{mon10theo}\right\}	
\end{equation}
and 
\begin{equation}
	\Lambda(w)=\left\{(y,p^i):\ (y,p^i)\ \mbox{is the state associated to a control}\ \bar{h}\in I(w)\ \mbox{and}\ (y,p^i)\ \mbox{satisfies}\ \eqref{y}\right\}.	
\end{equation}
In this way, we introduce a multivalued mapping
$$
w\longmapsto \Lambda(w).
$$
We want to prove that this mapping has a fixed point $y$. Of course, this will imply that there exists a control $\bar{h}\in L^2(\omega_T)$ such that the solution of \eqref{yop}-\eqref{pop} satisfies \eqref{mainobj}.

To this end, we will use the Kakutani's fixed point Theorem that can be applied on $\Lambda$. Proceeding as in \cite[Theorem 1.3]{danynina2021} or \cite[Theorem 1.1]{birba2016}, we can prove the following properties for every $w\in W$:

\begin{itemize}
	\item $\Lambda(w)$ is a non empty, closed and convex set of $W$.
	\item $\Lambda(w)$ is a bounded and compact set of $W$.
	\item The application $w\longmapsto\Lambda(w)$ is upper hemi-continuous. 
\end{itemize}
This end the proof of Theorem \ref{theolinear} and furthermore the proof of null controllability of system \eqref{eq}.	\hfill $\blacksquare$

\section{Conclusion remarks}  \label{conclusion}

We proved the Stackelberg-Nash null controllability of a coupled degenerate non linear parabolic equations with one leader and two followers. Since our functionals are not convex because the system is non linear, we considered first the Nash quasi-equilibrium. In the first time, for each leader fixed, we have proved the existence, uniqueness and characterization of Nash quasi-equilibrium. Then using some hypothesis, we have showed the equivalence between Nash quasi-equilibrium and Nash equilibrium. By suitable Carleman estimates, we have established an observability inequality which is the key to deduce our controllability result. As future direction, we will extend the results obtained in this paper to a more general system of $m$ cascade coupled parabolic degenerate PDEs as in \cite{Burgos2010}.

\section*{Acknowledgments} 
We would like to thank the reviewers and the editors for their valuable comments and suggestions which helped us to improve significantly the paper.

\section*{Funding}
The first author was supported by the German Academic Exchange Service (D.A.A.D) under the Scholarship Program PhD AIMS-Cameroon. The second  author was supported by a grant from the African Institute for Mathematical Sciences, www.nexteinstein.org, with financial support from the Government of Canada, provided through Global Affairs Canada, www.international.gc.ca, and the International Development Research Centre, www.idrc.ca.

\end{document}